\documentclass[12pt,english]{article}
\usepackage[T1]{fontenc}
\usepackage[latin9]{inputenc}
\usepackage{geometry}
\usepackage{babel}
\usepackage{float}
\usepackage{units}
\usepackage{amsmath}
\usepackage{amsthm}
\usepackage{amssymb}
\usepackage{setspace}
\PassOptionsToPackage{normalem}{ulem}
\usepackage{ulem}
\usepackage[unicode=true,pdfusetitle,
 bookmarks=true,bookmarksnumbered=false,bookmarksopen=false,
 breaklinks=false,pdfborder={0 0 1},backref=page,colorlinks=false]
 {hyperref}
\hypersetup{
 linkcolor=blue,citecolor=blue}

\makeatletter

\floatstyle{ruled}
\newfloat{algorithm}{tbp}{loa}
\providecommand{\algorithmname}{Algorithm}
\floatname{algorithm}{\protect\algorithmname}

\numberwithin{equation}{section}
\numberwithin{figure}{section}
\theoremstyle{plain}
\newtheorem{thm}{\protect\theoremname}[section]
\theoremstyle{definition}
\newtheorem{defn}[thm]{\protect\definitionname}
\theoremstyle{definition}
\newtheorem{example}[thm]{\protect\examplename}
\theoremstyle{plain}
\newtheorem{prop}[thm]{\protect\propositionname}
\ifx\proof\undefined
\newenvironment{proof}[1][\protect\proofname]{\par
	\normalfont\topsep6\p@\@plus6\p@\relax
	\trivlist
	\itemindent\parindent
	\item[\hskip\labelsep\scshape #1]\ignorespaces
}{%
	\endtrivlist\@endpefalse
}
\providecommand{\proofname}{Proof}
\fi
\theoremstyle{plain}
\newtheorem{cor}[thm]{\protect\corollaryname}
\theoremstyle{remark}
\newtheorem{rem}[thm]{\protect\remarkname}
\theoremstyle{remark}
\newtheorem{claim}[thm]{\protect\claimname}
\theoremstyle{plain}
\newtheorem{lem}[thm]{\protect\lemmaname}

\usepackage{xcolor}
\hypersetup{
    colorlinks,
    linkcolor={red!50!black},
    citecolor={blue!50!black},
    urlcolor={blue!80!black}
}

\usepackage{algorithm,algpseudocode}

\makeatother

\providecommand{\claimname}{Claim}
\providecommand{\corollaryname}{Corollary}
\providecommand{\definitionname}{Definition}
\providecommand{\examplename}{Example}
\providecommand{\lemmaname}{Lemma}
\providecommand{\propositionname}{Proposition}
\providecommand{\remarkname}{Remark}
\providecommand{\theoremname}{Theorem}

\begin{document}
\title{Primitivity Testing in Free Group Algebras via Duality}
\author{Matan Seidel, Danielle Ernst-West, and Doron Puder}
\maketitle
\begin{abstract}
Let $K$ be a field and $F$ a free group. By a classical result of
Cohn and Lewin \cite{Cohn1964,Lewin1969}, the free group algebra
$K\left[F\right]$ is a \emph{free ideal ring} (FIR): a ring over
which the submodules of free modules are themselves free, and of a
well-defined rank. Given a finitely generated right ideal $I\leq K\left[F\right]$
and an element $f\in I$, we give an explicit algorithm determining
whether $f$ is part of some basis of $I$. More generally, given
free $K[F]$-modules $M\le N$, we provide algorithms determining
whether $M$ is a free summand of $N$, and whether $N$ admits a
free splitting relative to $M$. These can also be used to obtain
analogous algorithms for free groups $H\le J$.

As an aside, we also provide an algorithm to compute the intersection
of two given submodules of a free $K\left[F\right]$-module.

A key feature of this work is the introduction of a duality, induced
by a matrix with entries in a free ideal ring, between the respective
algebraic extensions of its column and row spaces. 
\end{abstract}
\tableofcontents{}

\section{Introduction}

Free groups form a central topic of research in group theory. A useful
approach in their study is to extend them linearly over a field (or
a commutative ring) to form the associated free group algebra (or
free group ring, respectively). Let $K$ be a fixed field and $F$
a free group on $r\in\mathbb{Z}_{\geq1}$ generators. Their associated
free group algebra, denoted $K\left[F\right]$, is the $K$-vector
space having $F$ as a basis, endowed with the product defined by
$\left(\sum_{u\in F}\alpha_{u}u\right)\cdot\left(\sum_{v\in F}\beta_{v}v\right)=\sum_{u,v\in F}\alpha_{u}\beta_{v}uv$.
The resulting $K$-algebra structure enables the application of tools
from linear algebra, ring theory and homological algebra, leading
to significant results such as the Stallings-Swan Theorem \cite{Stallings1968,Swan1969},
which characterizes free groups by their cohomological dimension.
Group algebras may also serve to solve open problems concerning the
underlying groups: e.g., the compressed-inert conjecture in free
groups in \cite{JaikinZapirain2024}.

The usefulness of free group algebras in this context is also due
to the similarities they share with free groups: In analogy to the
classical Nielsen-Schreier Theorem, which guarantees that subgroups
of free groups are themselves free, a theorem due to Cohn \cite{Cohn1964}
and Lewin \cite{Lewin1969} shows that one-sided ideals in the free
group algebra (over any field) are free as modules over it\footnote{\label{fn:Cohn Vs. Lewin}Some claim that the first correct proof
of this result is due to Lewin -- see \cite[Footnote 5]{HogAngeloni2006}.}. However, while many well-known tools and algorithms exist for free
groups --- most notably Whitehead's algorithms for detecting primitive
words and automorphic equivalence \cite{Whitehead1936,Whitehead1936a}
and Stallings' core graphs \cite{Stallings1983} --- the development
of analogous tools for free group algebras has been more limited.
A significant contribution in this area is an algorithm due to Rosenmann
\cite{Rosenmann1993}, which constructs both a free basis and a Gröbner
basis for finitely generated one-sided ideals in free group algebras.
The current work fills some of these gaps in the realm of free group
algebras: it presents an analog of Whitehead's first algorithm, as
well as an algorithm for finding the intersection of two one-sided
ideals. 

This paper arose from concrete computational needs relating to the
theory of word measures: measures on compact groups induced by words
in free groups. The study of word measures on various families of
compact groups exposed interesting phenomena and deep structure (e.g.\ \cite{Puder2015,Magee2019,Magee2024,puder2023stable}).
In a recent paper by the current authors \cite{ErnstWest2024}, we
exposed the relation between free group algebras and word measures
on matrix groups over finite fields. More specifically, the existence
and number of certain ideals in the free group algebra containing
some element imprimitively (namely, so that this element does not
belong to any basis -- see Definition \ref{def: primitivity, free factor, algebraic, algebraic-free decomposition})
turn out to control certain moments of these word measures \cite[Thm.~1.4 and Conj.~1.6]{ErnstWest2024}.
In particular, the following question appears in \cite[Sec.~5 and 7]{ErnstWest2024}:
Is there an algorithm for determining the primitivity of a given element
in a given ideal of a free group algebra? The algorithms we provide
in the current paper resolve this question and related ones.

\subsection{Extensions of Free Modules over Free Ideal Rings}

Throughout the article, the term \textquotedbl module\textquotedbl{}
refers to either left or right modules, unless an explicit orientation
is indicated. Given an $R$-module $M$ and elements $f_{1},f_{2},...,f_{t}\in M$,
the submodule they generate will be denoted $\sum_{i=1}^{t}f_{i}R$
if $M$ is a right $R$-module and $\sum_{i=1}^{t}Rf_{i}$ if $M$
is a left $R$-module. When the ambient module $M$ and its orientation
are clear from context, we will also use the notation $\left(f_{1},f_{2},...,f_{t}\right)$
for the generated submodule.

A \emph{free ideal ring }or \emph{FIR} is a ring $\mathcal{A}$ whose
one-sided ideals, when considered as $\mathcal{A}$-modules, are free
and have a well-defined rank (see Cohn \cite{Cohn2006} for an extensive
background). In general, a ring may satisfy the condition on one side
but not the other \cite[Sec.~2.10]{Cohn2006}. While some authors
use the term \emph{FIR} to mean a one-sided condition, we follow Cohn
and require the condition on both sides. Some examples of free ideal
rings include fields, the ring of integers $\mathbb{Z}$, and free
associative algebras over a field \cite[Cor.~2.5.2]{Cohn2006}. Free
group algebras form yet another example by the previously discussed
result of Cohn and Lewin (see also \cite{HogAngeloni2006} for a short
and beautiful later proof by Hog-Angeloni).

Over a free ideal ring, any submodule of a free module is also free
\cite[Thm.~1.5.3]{Cartan1999}. This compels us to adapt the following
definitions from free groups (the original definitions for free groups
appear, for example, in \cite{Miasnikov2007} and \cite{Puder2015},
as well as in Section \ref{sec: algs for free groups} below):
\begin{defn}
\label{def: primitivity, free factor, algebraic, algebraic-free decomposition}Let
$\mathcal{A}$ be a free ideal ring, $N$ a free $\mathcal{A}$-module
and $M$ a submodule of $N$.
\begin{enumerate}
\item An element $f\in N$ is said to be \uline{primitive} in $N$ if
it is part of some basis of $N$. 
\item The extension $M\leq N$ is called \uline{free} and $M$ is called
a \uline{free facto}r of $N$ if a basis of $M$ can be extended
to a basis of $N$. In this case we write $M\leq_{*}N.$
\item The extension $M\leq N$ is called \uline{algebraic} if it admits
no intermediate proper free factor $M\leq L\lneqq_{*}N.$ In this
case we write $M\leq_{\text{alg}}N.$
\item If at least one of $M$ or $N$ is of finite rank, the \uline{algebraic
closure} of the extension $M\leq N$ is the unique free factor of
$N$ which extends $M$ algebraically (see Theorem \ref{Theorem: existence and uniqueness of Algebraic - free decomposition}
for a proof of existence and uniqueness, as well as for some equivalent
definitions).\footnote{The terms we use here are known under various names, depending on
context. Free factors are usually called \emph{direct summands} in
ring theory. In Cohn's book ``Free Ideal Rings'' (\cite{Cohn2006}),
the algebraic closure of $M\leq N$ is simply called the \emph{closure}
of $M$ in $N$, and if $M\leq_{\text{alg}}N$ then $M$ is called
a \emph{dense} submodule of $N$. If $M\le N$ is not an algebraic
extension, one could also say that $N$ \emph{admits a free splitting
relative to} $M$.}
\end{enumerate}
\end{defn}
Our first result is the following:
\begin{thm}
\label{thm: There exists algorithms for alg-free decomp}Given an
extension $M\leq N$ of finitely generated $K\left[F\right]$-submodules
of $K\left[F\right]^{m}$ for some $m\in\mathbb{Z}_{\geq1}$, there
exist algorithms for computing the extension's algebraic closure (Algorithm
\ref{alg: Get-Algebraic-Free}), and consequently for testing if the
extension is algebraic (Algorithm \ref{alg: Is-Algebraic-Extension})
and if $M$ is a free factor of $N$ (Algorithm \ref{alg: Is-Free-Extension}).
\end{thm}
If $\text{rk}F=1$, Theorem \ref{thm: There exists algorithms for alg-free decomp}
holds due to well-known normal forms: see Remark \ref{rem: In rkF=00003D1 coefficients can be extracted in a different way}.

\subsection{Extensions of Free Groups\label{subsec:introducing algo for free group}}

Given an extension of finitely generated free groups, there exist
analogous notions as in Definition \ref{def: primitivity, free factor, algebraic, algebraic-free decomposition}
for the algebraicity and freeness of the extension, as well as for
its algebraic closure.  Moreover, for free groups, there are known
algorithms -- most famously Whitehead's -- for detecting these properties
and performing related computations (See Section \ref{sec: algs for free groups}
for further details and precise definitions).

Let $F$ be a free group. Using some fixed field $K$, one may associate
to every subgroup $H\leq F$ a (right) ideal $J_{H}$ in the free
group algebra $K\left[F\right]$, called its \emph{augmentation ideal}
(following Cohen's \cite[Sec.~4]{Cohen2006}, and see Section \ref{sec: algs for free groups}
for details).

Our second result is that, given an extension $H\leq H'$ of finitely
generated subgroups of $F$, by applying the algorithms of Theorem
\ref{thm: There exists algorithms for alg-free decomp} to the extension
$J_{H}\leq J_{H'}$ in $K\left[F\right]$, one obtains new algorithms
for computing the algebraic closure of the extension $H\leq H'$ (Algorithm
\ref{alg:Get-Algebraic-Free-For-Groups}), and consequently for testing
if the extension is algebraic (Algorithm \ref{alg:Is-Algebraic-Extension-For-Group})
and if $H$ is a free factor of $H'$ (Algorithm \ref{alg:Is-Free-Extension-For-Groups}).

\subsection{A Duality Induced by a Matrix over a Free Ideal Ring\label{subsec:A-Duality-intro}}

Let $\mathcal{A}$ be again a free ideal ring and $Q\in\text{Mat}_{k\times m}\left(\mathcal{A}\right)$
a $k\times m$ matrix with entries in $\mathcal{A}$. The \emph{row
space} of $Q$, denoted $L_{Q}$ ($L$ for 'left'), is the submodule
of the \uline{left} $\mathcal{A}$-module $\mathcal{A}^{m}$ generated
by the rows of $Q$, so $L_{Q}=\left\{ vQ\,\mid\,v\in\text{Mat}_{1\times k}\left(\mathcal{A}\right)\right\} $.
Similarly, the \emph{column space} of $Q$, denoted $R_{Q}$ ($R$
for 'right'), is the submodule of the \uline{right} $\mathcal{A}$-module
$\mathcal{A}^{k}$ generated by the columns of $Q$, so $R_{Q}=\left\{ Qv\,\mid\,v\in\text{Mat}_{m\times1}\left(\mathcal{A}\right)\right\} $.

A key ingredient in the algorithms of Theorem \ref{thm: There exists algorithms for alg-free decomp}
is the introduction of a duality, associating to each extension $M$
of the column (respectively, row) space of $Q$ an extension $M^{*Q}$
of its row (respectively, column) space, which we call the \emph{$Q$-dual}
of $M$ (see Section \ref{sec:The-Duality-Induced} for precise definitions).
Our next result establishes the usefulness of this duality in obtaining
the algebraic closure:
\begin{thm}
\label{Theorem: algebraic-free decomposition-1}(\textbf{Algebraic
closure via duality)} Let $Q\in\text{Mat}{}_{k\times m}\left(\mathcal{A}\right)$.
If $R_{Q}\leq M\leq\mathcal{A}^{k}$ then the double $Q$-dual of
$M$ is the algebraic closure of $R_{Q}$ in $M$. Similarly, if $L_{Q}\leq N\leq\mathcal{A}^{m}$
then the double $Q$-dual of $N$ is the algebraic closure of $L_{Q}$
in $N$. Namely, 
\[
R_{Q}\leq_{\text{alg}}\left(M^{*Q}\right)^{*Q}\leq_{*}M\,\,\,\,\,\text{and}\,\,\,\,\,L_{Q}\leq_{\text{alg}}\left(N^{*Q}\right)^{*Q}\leq_{*}N.
\]
\end{thm}
In particular, a double application of the $Q$-dual can be used to
determine whether a given extension of free $\mathcal{A}$-modules
is free and whether it is algebraic.

The $Q$-dual defined here is a variation on the duality of Cohn \cite[Chap.~5.2]{Cohn2006},
albeit with a different goal in mind. The duality of Cohn is defined
on quotient modules arising from an algebraic extension of free $\mathcal{A}$-modules.
In contrast, the $Q$-dual defined here is a mapping between free
submodules extending fixed base modules (the column/row spaces of
$Q$), and is defined more generally for extensions which are not
necessarily algebraic. It is exactly this possible lack of algebraicity
and the way it is reflected in the properties of the $Q$-dual which
come into focus here, an idea exemplified by our next result:
\begin{thm}
\label{Thm: Equivalent conditions of algebraic extensions via duality in Introduction}Let
$Q\in\text{Mat}{}_{k\times m}\left(\mathcal{A}\right)$ and let $M_{0}\leq M$
be an extension of free $\mathcal{A}$-modules, where either $M_{0}$
is the row space of $Q$ and $M\leq\mathcal{A}^{m}$ or $M_{0}$ is
the column space of $Q$ and $M\leq\mathcal{A}^{k}.$ Then the following
are equivalent:
\begin{enumerate}
\item $M_{0}\leq_{\text{alg}}M.$
\item $\left(M^{*Q}\right)^{*Q}=M$.
\item $rk$$\left(M^{*Q}\right)=rk\left(M\right)$.
\end{enumerate}
\end{thm}

\subsection{Duality and Word Measures}

The $Q$-duality is useful also for the theory of word measures on
the family $\left\{ GL_{N}\left(K\right)\right\} _{N\in\mathbb{N}}$
where $K$ is a finite field of order $q$. Fixing a word $w\in F$
where $F$ is the free group on $r$ generators, a $w$-random element
of the group $GL_{N}\left(K\right)$ is obtained by sampling independently
$r$ uniformly random elements $g_{1},g_{2},...,g_{r}\in GL_{N}\left(K\right)$
and evaluating $w\left(g_{1},g_{2},...,g_{r}\right).$ The expected
number of fixed vectors in $K^{N}$ of a $w-$random matrix in $GL_{N}\left(K\right)$,
denoted $\mathbb{E}_{w}^{\text{GL}_{N}\left(K\right)}\left[\text{fix}\right]$,
is studied in \cite{ErnstWest2024} and is shown to be a rational
function in $q^{N}$. Moreover, the order of magnitude of the quantity
$\mathbb{E}_{w}^{\text{GL}_{N}\left(K\right)}\left[\text{fix}\right]-2$,
which reflects the deviation from the uniform measure, is related
to the set of algebraic extensions of the ideal $\left(w-1\right)K\left[F\right]$
in the free group algebra $K\left[F\right]$. Section \ref{sec:The-Duality in the free group algebra}
specializes the $Q$-duality to this case. Relying on a right-left
inversion $\iota$ inside the free group algebra (see Definition \ref{def: iota}),
we prove the existence of a natural rank-preserving involution on
the algebraic extensions of the ideal $\left(w-1\right)K\left[F\right]$:
\begin{thm}
\label{thm: there exists an involution on extensions of w-1}Let $w\in F$
and $Q$ be the $1\times1$ matrix over $K\left[F\right]$ with $w-1$
as its single entry. The map $I\mapsto\iota\left(I^{*Q}\right)$ is
an involution on algebraic extensions of the ideal $\left(w-1\right)K\left[F\right]$,
preserving rank and inverting the order of inclusion.
\end{thm}
For example, for $w=xyx^{-1}y^{-1}\in\left\langle x,y\right\rangle =F$,
the involution of Theorem \ref{thm: there exists an involution on extensions of w-1}
couples the ideal $\left(w-1\right)K\left[F\right]$ with the improper
ideal $K\left[F\right]$ and every ideal of the form $\left(x-\alpha\right)K\left[F\right]+\left(y-\beta\right)K\left[F\right]$
where $\alpha,\beta\in K^{*}$ with the ideal $\left(x-\alpha^{-1}\right)K\left[F\right]+\left(y-\beta^{-1}\right)K\left[F\right]$.

The theory developed in \cite{ErnstWest2024} connecting word measures
on $\text{GL}_{N}(K)$ and free group algebras is inspired by a similar
theory connecting word measures on the symmetric groups $S_{N}$ and
the poset of subgroups of the free group \cite{Puder2014,Puder2015,Hanany2023}.
The main result of \cite{Puder2015} gives the precise order of magnitude
of $\mathbb{E}_{w}^{S_{N}}[\text{fix}]-1$, where $\mathbb{E}_{w}^{S_{N}}[\text{fix}]$
is the expected number of fixed points in a $w$-random permutation
in $S_{N}$. A key ingredient in the proof of that result was the
definition of a function on extensions of subgroups of $F$ which
takes the value $\mathbb{E}_{w}^{S_{N}}[\text{fix}]$ on the extension
$\langle w\rangle\le F$. Conjecture 1.6 from \cite{ErnstWest2024}
deals, similarly, with the precise order of magnitude of $\mathbb{E}_{w}^{\text{GL}_{N}\left(K\right)}\left[\text{fix}\right]-2$
in $\text{GL}_{N}(K)$. In light of the proof in the case of $S_{N}$,
the following definition may be of value in attempts to prove this
conjecture:
\begin{defn}
\label{def: Big phi}Let $K$ be a finite field of order $\left|K\right|=q$,
and let $L\leq M$ an extension of free $K\left[F\right]$-modules
such that $\text{rk}M<\infty$. Define $\phi_{L,M}(N):=q^{N}\cdot\mathbb{P}\left(L\subseteq\ker\varphi_{M}\right)$,
where the probability is taken over a $K\left[F\right]$-module structure
on $K^{N}$ and a $K\left[F\right]$-module homomorphism $\varphi_{M}:M\rightarrow K^{N}$,
chosen uniformly at random.

Note that the number of possible choices in Definition \ref{def: Big phi}
is indeed finite since $K$ is finite and $\text{rk}M<\infty$ (see
Section \ref{sec: Theorem about Phi and Duality} for more details),
and that $\phi_{\left(w-1\right)K\left[F\right],K\left[F\right]}(N)=\mathbb{E}_{w}^{\text{GL}_{N}\left(K\right)}\left[\text{fix}\right]$.
The following result illustrates both the usefulness of the duality
and its well-behaved nature with respect to other natural notions.
It shows that the involution arising from the $Q$-duality in Theorem
\ref{thm: there exists an involution on extensions of w-1} leaves
the probability above invariant:
\end{defn}
\begin{thm}
\label{Thm: Phi invariant under duality}Let $Q$ be a matrix over
$K\left[F\right]$ and $L$ and $M$ algebraic extensions of its column
space such that $L\leq M$. Then $\phi_{L,M}=\phi_{M^{*Q},L^{*Q}}.$
\end{thm}
(The analogous statement about algebraic extensions of the row space
is equivalent, by applying the theorem on duals.) As $M$ is a free
$K\left[F\right]$-module, a random homomorphism from $M$ to $K^{N}$
can be obtained by selecting uniformly at random the images of some
fixed basis of $M$. In light of this, Theorem \ref{Thm: Phi invariant under duality}
may seem somewhat surprising at first glance in the sense that the
ranks of $M$ and $L^{*Q}$ are generally not equal, so the two probabilities
which coincide are taken over a different number of uniformly chosen
vectors in $K^{N}$. We illustrate this point in the following example: 
\begin{example}
\label{Example: calculation of phi via duality}Let $K$ be a field
of order $q$, where $q$ is a prime power, and let $w=xyx^{-1}y^{-1}$
be the commutator in $F=\left\langle x,y\right\rangle $. We calculate
$\phi_{I,J}\left(N\right)$ for the right ideals $I=\left(w-1\right)K\left[F\right]$
and $J=\left(x-1\right)K\left[F\right]+\left(y-1\right)K\left[F\right]$.
We first choose independently and uniformly at random two matrices
$X,Y\in GL_{N}\left(K\right)$ to give $K^{N}$ a right $K\left[F\right]$-module
structure, defined by letting $x$ and $y$ act on row vectors of
$K^{N}$ by right multiplication with $X$ and $Y$ respectively.
Approaching the calculation directly, we write 
\[
w-1=\left(x-1\right)\left(yx^{-1}y^{-1}-x^{-1}y^{-1}\right)+\left(y-1\right)\left(x^{-1}y^{-1}-y^{-1}\right),
\]
so $\phi_{I,J}\left(N\right)$ is $q^{N}$ times the probability that
two uniformly chosen row vectors $v_{1},v_{2}\in K^{N}$ satisfy $0=v_{1}\left(YX^{-1}Y^{-1}-X^{-1}Y^{-1}\right)+v_{2}\left(X^{-1}Y^{-1}-Y^{-1}\right)$.
While this calculation can be completed, we proceed with a second
(and easier) approach, using the $Q$-duality with respect to the
$1\times1$ matrix $Q=\left(w-1\right)$. The $Q$-duals are $I^{*Q}=K\left[F\right]$
and $J^{*Q}=J$ (note that $J$ is a two-sided ideal), so by Theorem
\ref{Thm: Phi invariant under duality} we have $\phi_{I,J}\left(N\right)=\phi_{J,K[F]}\left(N\right)$.
Letting the uniformly chosen $X,Y\in GL_{N}\left(K\right)$ act on
column vectors of $K^{N}$ by left multiplication, the latter is $q^{N}$
times the probability that a single uniformly chosen column vector
$v\in K^{N}$ is fixed by both $X$ and $Y$. By splitting into the
cases that $v=0$ and $v\neq0$ we obtain
\[
\phi_{I,J}\left(N\right)=q^{N}\left[\frac{1}{q^{N}}\cdot1+\frac{q^{N}-1}{q^{N}}\cdot\left(\frac{1}{q^{N}-1}\right)^{2}\right]=1+\frac{1}{q^{N}-1}.
\]
\end{example}

\subsection{Paper Organization and Notation}

The paper aims to present the algorithms mentioned in Theorems \ref{thm: There exists algorithms for alg-free decomp}
and Section \ref{subsec:introducing algo for free group} in a straightforward
manner and without assuming prior knowledge of free ideal rings or
homological algebra. It is structured accordingly, with the required
ideas presented or recalled first, leading up to the explicit algorithms
for free group algebras in Section \ref{sec: Algs for group algebras}
and for free groups in Section \ref{sec: algs for free groups}. 

Section \ref{sec:Algebraic-and-Free} is devoted to extensions of
free modules over a free ideal ring, mainly proving the existence
and uniqueness of the algebraic closure in Theorem \ref{Theorem: existence and uniqueness of Algebraic - free decomposition}.
Section \ref{sec:The-Duality-Induced} develops the duality induced
by a matrix with entries in a free ideal ring and its properties,
and proves Theorems \ref{Theorem: algebraic-free decomposition-1}
and \ref{Thm: Equivalent conditions of algebraic extensions via duality in Introduction}.
Section \ref{sec:The-Duality in the free group algebra} specializes
the duality of the previous section to free group algebras, proving
Theorem \ref{thm: there exists an involution on extensions of w-1}.
Section \ref{sec: Theorem about Phi and Duality} proves Theorem \ref{Thm: Phi invariant under duality}.
Section \ref{sec: Algs for group algebras} recalls the relevant aspects
of Rosenmann's algorithm for computing a basis for right-ideals in
free group algebras and provides the explicit algorithms of Theorem
\ref{thm: There exists algorithms for alg-free decomp}. Section \ref{sec: algs for free groups}
is devoted to the analogous algorithms in free groups. As an aside,
Section \ref{sec: intersection algorithm} describes an algorithm
(Algorithm \ref{alg:Intersection_of_modules}) for computing a generating
set for the intersection of free $K\left[F\right]$-modules. We end
with a few open questions in Section \ref{sec:Further-Algorithms?}.

The ring in which we work is assumed throughout the paper to be a
free ideal ring and is denoted by $\mathcal{A}$. From Section \ref{sec:The-Duality in the free group algebra}
onwards, $\mathcal{A}$ is further assumed to be a free group algebra.
For clarity, we use occasionally the more explicit notation $K\left[F\right]$
when more than one free group algebra is involved.

\section{\label{sec:Algebraic-and-Free}Algebraic and Free Extensions in a
Free Ideal Ring}

In this section we prove several basic facts regarding free modules
over a free ideal ring, most importantly proving the existence and
uniqueness of the algebraic closure in Theorem \ref{Theorem: existence and uniqueness of Algebraic - free decomposition}.
Throughout this section $\mathcal{A}$ denotes a free ideal ring.
Recall from Definition \ref{def: primitivity, free factor, algebraic, algebraic-free decomposition}
the notions of a free factor and an algebraic extension. We start
with an elementary observation. 
\begin{prop}
\label{prop:equivalent characterizations of free factor}Let $N$
be a free $\mathcal{A}$-module and $M\leq N$. Then the following
are equivalent:
\begin{enumerate}
\item The submodule $M$ is a free factor of $N$.
\item The quotient module $N/M$ is free.
\item There exists a retraction $N\rightarrow M,$ i.e., an $\mathcal{A}$-module
homomorphism $\tau:N\rightarrow M$ such that $\tau\mid_{M}=Id_{M}.$
\end{enumerate}
\end{prop}
\begin{proof}
\uline{(1=>2):} If $M\leq_{*}N$ then there exists some $M'\leq N$
such that $N=M\oplus M',$ so $N/M\cong M'$ is free since $M'$ is
a submodule of the free $\mathcal{A}$-module $N\text{. }$\\
\uline{(2=>3):} Consider the canonical short exact sequence $0\rightarrow M\rightarrow N\overset{\pi}{\rightarrow}N/M\rightarrow0$.
If $N/M$ is free then there exists an $\mathcal{A}$-module homomorphism
$s:N/M\rightarrow N$ such that $\pi\circ s=Id_{N/M}$, so $N=M\oplus s\left(N/M\right).$
The desired $\tau$ is the projection onto the first summand.\\
\uline{(3=>1):} If a retraction $\tau:N\rightarrow M$ exists then
$N=M\oplus\ker\tau.$ Since $N$ is a free module over the free ideal
ring $\mathcal{A}$, its submodules $M$ and $\ker\tau$ are free
as well, and hence admit bases. Any basis of $M$ now extends to a
basis of $N$ by adding to it the elements of a basis of $\ker\tau$.
\end{proof}
We next show how the properties of freeness and algebraicity are preserved
under certain submodule operations:
\begin{prop}
\label{Intersection preserves freeness of extension}Let $L$, $M$
and $N$ be submodules of a free $\mathcal{A}$-module such that $M\leq_{*}N$.
Then $M\cap L\leq_{*}N\cap L.$
\end{prop}
\begin{proof}
By Proposition \ref{prop:equivalent characterizations of free factor},
the quotient module $N/M$ is a free $\mathcal{A}$-module, thus so
is its submodule $\nicefrac{\left((N\cap L)+M\right)}{M}$. By the
second isomorphism theorem the latter module is isomorphic to $\nicefrac{N\cap L}{M\cap L}$,
which by Proposition \ref{prop:equivalent characterizations of free factor}
again implies that $M\cap L\leq_{*}N\cap L.$
\end{proof}
Applying Proposition \ref{Intersection preserves freeness of extension}
inductively and using the transitivity of being a free factor, we
obtain:
\begin{cor}
\label{Finite intersection of free factors is a free factor} A finite
intersection of free factors of a free $\mathcal{A}$-module $N$
is itself a free factor of $N$.
\end{cor}
The following observation is not used in the paper but we add it here
for completeness.
\begin{prop}
\label{Sum of algebraic is algebraic}Let $L$, $M$ and $N$ be submodules
of a free $\mathcal{A}$-module. If $M$ and $N$ are algebraic extensions
of $L$ then so is $M+N$.
\end{prop}
\begin{proof}
Let $L\leq T\leq_{*}M+N$ be a free factor of $M+N$ containing $L$.
Intersecting with $M$, we obtain from Proposition \ref{Intersection preserves freeness of extension}
that $L\leq T\cap M\leq_{*}M$ which implies by algebraicity of $M$
that $T\cap M=M$. Intersecting similarly with $N$ we get $M,N\subseteq T$
and so $M+N\subseteq T$, as needed.
\end{proof}
\begin{prop}
\label{every extension has a finite-rank free factor} Let $N$ be
a free $\mathcal{A}$-module. Then for every finitely generated submodule
$M\leq N$ there exists a finitely generated free factor of $N$ containing
$M$.
\end{prop}
\begin{proof}
Let $C$ be a basis for $N$ and $B$ a finite basis for $M$. Each
$b\in B$ lies in $N$ and is, therefore, expressible as an $\mathcal{A}$-linear
combination in some finite subset $C_{b}\subseteq C$. Let $N'\leq N$
be the submodule generated by the finite subset $C':=\bigcup_{b\in B}C_{b}$
of $C$. Clearly, $N'$ satisfies the required properties.
\end{proof}
We are now ready to show the fact mentioned in Definition \ref{def: primitivity, free factor, algebraic, algebraic-free decomposition}:
the existence and uniqueness of the algebraic closure of any extension
of free $\mathcal{A}$-modules $M\leq N$, with at least one of $M$
or $N$ of finite rank. This is analogous to the case of free groups
(see \cite[Thm. 3.16]{Miasnikov2007}). 
\begin{thm}
\textbf{\label{Theorem: existence and uniqueness of Algebraic - free decomposition}(The
algebraic closure)} Let $N$ be a free $\mathcal{A}$-module and $M\leq N$
a submodule such that at least one of $M$ or $N$ is finitely generated.
Then there exists a unique algebraic extension $L$ of $M$ which
is a free factor of $N$, i.e., such that $M\leq_{\text{alg}}L\leq_{*}N$.
Furthermore, $L$ has the following equivalent characterizations:
\begin{enumerate}
\item $L=\bigcap_{M\leq L'\leq_{*}N}L'.$
\item $L=\bigcup_{M\leq_{\text{alg}}L'\leq N}L'$.
\item $L$ is the set of elements $x\in N$ such that $\varphi(x)=0$ for
every ${\cal A}$-module homomorphism $\varphi:N\to{\cal A}$ vanishing
on $M$. Namely, 
\[
L=\bigcap_{\substack{\varphi:N\rightarrow\mathcal{A}\,\text{such\,that}\\
\varphi\mid_{M}\equiv0
}
}\ker\varphi.
\]
\end{enumerate}
\end{thm}
\begin{proof}
Let $X$ be the set of free factors of $N$ which contain $M$. Observe
that $X$ contains an element of finite rank: if $N$ is finitely
generated then $N\in X$ and if $M$ is finitely generated then this
is guaranteed by Proposition \ref{every extension has a finite-rank free factor}.
Let $L$ be an element of $X$ of minimal rank. This $L$ is a free
factor of $N$ by definition, and an algebraic extension of $M$:
otherwise, there would be some $M\leq L'\lneqq_{*}L\leq_{*}N$, contradicting
the minimality of $\text{rk}L$.
\begin{enumerate}
\item Suppose there exists some intermediate free factor $M\leq L'\leq_{*}N$
which does not contain $L$. Then $L'\cap L$ is a proper free factor
of $L$, and is thus of smaller rank, contradicting the minimality
of $\text{rk}L$ in $X$.
\item Let $M\leq_{\text{alg}}L'\leq N$. Intersecting the free extension
$L\leq_{*}N$ with $L'$ gives $M\leq L'\cap L\leq_{*}L'$, which
by the algebraicity of $M\leq_{\text{alg}}L'$ implies $L'\cap L=L'$,
so $L'\subseteq L$.
\item On the one hand, for every $\varphi:N\rightarrow\mathcal{A}$ we have
$N/\ker\varphi\cong Im\varphi$ is a free $\mathcal{A}$-module (as
a submodule of $\mathcal{A}$) and so by Proposition \ref{prop:equivalent characterizations of free factor}
$\ker\varphi\leq_{*}N$. Thus,
\[
\bigcap_{\substack{\varphi:N\rightarrow\mathcal{A}\\
\varphi\mid_{M}\equiv0
}
}\ker\varphi\supseteq\bigcap_{M\leq L'\leq_{*}N}L'=L.
\]
On the other hand, let $x\in N\backslash L$. Since the quotient $N/L$
is a free $\mathcal{A}$-module (as $L\leq_{*}N$) and the image of
$x$ in $N/L$ is nonzero, there exists a projection onto some coordinate
$p_{x}:N/L\rightarrow\mathcal{A}$ such that $p_{x}\left(\pi\left(x\right)\right)\neq0,$
where $\pi:N\rightarrow N/L$ is the canonical projection. This shows
that $x\notin\ker\left(p_{x}\circ\pi\right)$ while $M\le\ker\left(p_{x}\circ\pi\right)$.
Hence $\bigcap_{\substack{\varphi:N\rightarrow\mathcal{A}\\
\varphi\mid_{M}\equiv0
}
}\ker\varphi\subseteq L$.
\end{enumerate}
The uniqueness follows from properties (1) and (2), which together
imply that any other intermediate $M\leq_{\text{alg}}L'\leq_{*}N$
both contains $L$ and is contained in it.
\end{proof}
\begin{rem}
The requirement that at least one of $M$ or $N$ be finitely generated
is necessary: otherwise the intersection $\bigcap_{M\leq L'\leq_{*}N}L'$
might not be a free factor of $N$. See \cite{Burns1977} for a group-theoretic
example which is adaptable to our case by considering the associated
augmentation ideals (as in Section \ref{sec: algs for free groups}
below).
\end{rem}

\section{\label{sec:The-Duality-Induced}The Duality Induced by a Matrix over
a Free Ideal Ring}

Let $\mathcal{A}$ be again a free ideal ring and $Q\in\text{Mat}_{k\times m}({\cal A})$
a matrix over ${\cal A}$ which is fixed throughout this section.
We now introduce the duality between extensions of the row and column
spaces of $Q$. The main goal is to show that it can be used for characterizing
algebraicity (Theorem \ref{Thm: Equivalent conditions of algebraic extensions via duality in Introduction})
as well as for finding the algebraic closure (Theorem \ref{Theorem: algebraic-free decomposition-1}).

Most statements in this section will be stated for $\mathcal{A}$-modules
(i.e., either left or right), but will be proven for right $\mathcal{A}$-modules
alone. This is justified by considering the opposite ring $\mathcal{A}^{op}$,
which is identical to $\mathcal{A}$ as an abelian (additive) group
but such that the matrix multiplication (and ordinary multiplication)
on it is defined by $A\cdot_{op}B=\left(B^{T}A^{T}\right)^{T}$. Since
right modules over $\mathcal{A}$ correspond to left modules over
$\mathcal{A}^{op}$ and vice versa, the ring $\mathcal{A}^{op}$ is
also a free ideal ring. Thus, statements about left $\mathcal{A}$-modules
will follow from their right-module analogues applied to $\mathcal{A}^{op}$.

Recall from Section \ref{subsec:A-Duality-intro} the definition of
the column space $R_{Q}\le{\cal A}^{k}$ and the row space $L_{Q}\le{\cal A}^{m}$
of $Q$. The column space of $Q$ can be defined equivalently as the
image of the right $\mathcal{A}$-module homomorphism $T_{Q}:\mathcal{\mathcal{A}}^{m}\rightarrow\mathcal{A}^{k}$
defined by $T_{Q}(v)=Qv$. The map $Q\mapsto T_{Q}$ is functorial
in the sense that if the matrix product $PQ$ is defined then $T_{PQ}=T_{P}\circ T_{Q}$.
As $Im\left(T_{P}\circ T_{Q}\right)\subseteq Im\left(T_{P}\right)$,
we obtain:
\begin{claim}
\label{matrix product Column space contained}For any pair of matrices
$A$ and $B$ over $\mathcal{A}$ such that the product $AB$ is defined
we have $L_{AB}\subseteq L_{B}$ and $R_{AB}\subseteq R_{A}$.
\end{claim}
If one of the matrices is invertible\footnote{Recall that a $t\times t$ matrix $D$ over an arbitrary ring is called
\emph{invertible} if there exists a $t\times t$ matrix $D'$ such
that $DD'=D'D=I$.}, we get the following strengthening.
\begin{claim}
\label{invertible matrix preserves column space}Let $A\in\text{Mat}{}_{k\times m}\left(\mathcal{A}\right).$
Then for every invertible $B\in\text{Mat}{}_{m\times m}\left(\mathcal{A}\right)$
we have $R_{AB}=R_{A}$ and for every invertible $C\in\text{Mat}{}_{k\times k}\left(\mathcal{A}\right)$
we have $L_{CA}=L_{A}$.
\end{claim}
\begin{proof}
On the one hand $R_{AB}\subseteq R_{A}$ and on the other $R_{A}=R_{AB\cdot B^{-1}}\subseteq R_{AB}.$
\end{proof}

\subsection{The Duality Induced by a Matrix}

We work in fixed ambient $\mathcal{A}$-modules -- the right $\mathcal{A}$-module
$\mathcal{A}^{k}$ and the left $\mathcal{A}$-module $\mathcal{A}^{m}$
-- and use the poset notation induced by the relation $\leq$ of
containment between submodules. For example, $[R_{Q},\infty)$ denotes
the set of submodules of $\mathcal{A}^{k}$ containing the column
space of $Q$. 

In this subsection we associate to each $M\in[R_{Q},\infty)$ (respectively,
$M\in[L_{Q},\infty)$) an element $M^{*Q}\in[L_{Q},\infty)$ (respectively,
$M^{*Q}\in[R_{Q},\infty)$) called its $Q$-dual. Recall that by Proposition
\ref{every extension has a finite-rank free factor}, each such $M\in[R_{Q},\infty)$
admits a finitely generated free factor $R_{Q}\le M_{\text{fin}}\le_{*}M$.
\begin{prop}
\label{prop: duality is well-defined before its definition}

\begin{enumerate}
\item Let $M\in[R_{Q},\infty)$. Choose some finitely generated free factor
$M_{\text{fin}}\leq_{*}M$ containing $R_{Q}$ and a basis $f_{1},f_{2},...,f_{t}$
for $M_{\text{fin}}$. Let $B\in\text{Mat}_{t\times m}\left(\mathcal{A}\right)$
be the unique matrix satisfying
\begin{equation}
Q=\begin{pmatrix}\uparrow & \uparrow &  & \uparrow\\
f_{1} & f_{2} & \dots & f_{t}\\
\downarrow & \downarrow &  & \downarrow
\end{pmatrix}B.\label{eq: duality definition for R_Q}
\end{equation}
Then the row space $L_{B}\leq\mathcal{A}^{m}$ is independent of the
choice of free factor and basis. Furthermore, it is finitely generated,
contains $L_{Q}$ and satisfies $\text{rk}L_{B}\leq\text{rk}M$.
\item Let $M\in[L_{Q},\infty)$. Choose some finitely generated free factor
$M_{\text{fin}}\leq_{*}M$ containing $L_{Q}$ and a basis $g_{1},g_{2},...,g_{t}$
for $M_{\text{fin}}$. Let $C\in\text{Mat}_{k\times t}\left(\mathcal{A}\right)$
be the unique matrix satisfying
\begin{equation}
Q=C\begin{pmatrix}\leftarrow & g_{1} & \rightarrow\\
\leftarrow & g_{2} & \rightarrow\\
 & \vdots\\
\leftarrow & g_{t} & \rightarrow
\end{pmatrix}.\label{eq: duality definition for L_Q}
\end{equation}
Then the column space $R_{C}\leq\mathcal{A}^{m}$ is independent of
the choice of free factor and basis. Furthermore, it is finitely generated,
contains $R_{Q}$ and satisfies $\text{rk}R_{C}\leq\text{rk}M$.
\end{enumerate}
\end{prop}
Equation (\ref{eq: duality definition for R_Q}) requires the $i$-th
column of $B$ to comprise of the unique coefficients in $\mathcal{A}$
expressing the $i$-th column of $Q$ in the basis $f_{1},f_{2},...,f_{t}$,
from which the existence and uniqueness of $B$ follows. A similar
argument using (\ref{eq: duality definition for L_Q}) shows the existence
and uniqueness of $C$. 

The submodule provided by Proposition \ref{prop: duality is well-defined before its definition}
is central to this paper.
\begin{defn}
\textbf{\label{Def: Q-dual}(Q-duality)} Let $M\in[R_{Q},\infty)$
(respectively, $M\in[L_{Q},\infty)$). The \emph{$Q$-dual of $M$},
denoted $M^{*Q}$, is the submodule $L_{B}$ of $\mathcal{A}^{m}$
(respectively, $R_{C}$ of $\mathcal{A}^{k}$) obtained in the procedure
described in Proposition \ref{prop: duality is well-defined before its definition}.
\end{defn}
\begin{proof}[Proof of Proposition \ref{prop: duality is well-defined before its definition}]
 Let $M\in[R_{Q},\infty)$ and let $M_{\text{fin}}$ be a finitely
generated free factor of $M$ containing $R_{Q}$. Denote the rank
of $M_{\text{fin}}$ by $t$. We first show non-dependence on the
chosen basis. Let $F$ and $F'$ be $k\times t$ matrices whose columns
form two bases for $M_{\text{fin}}$. Expressing each basis with the
other, there exist matrices $D,D'\in\text{Mat}{}_{t\times t}\left(\mathcal{A}\right)$
such that $F=F'D$ and $F'=FD'$. But now $F=FD'D$ and $F'=F'DD'$,
and since the columns of $F,F'$ are bases, this implies that $DD'=D'D=I$,
so $D$ is invertible. Write $Q=FB$ and $Q=F'B'$ as in (\ref{eq: duality definition for R_Q}).
Then $F'B'=FB=F'DB$ and so $B'=DB,$ which implies by Claim \ref{invertible matrix preserves column space}
that $L_{B'}=L_{B}$. 

We next show the non-dependence on the chosen free factor. Let $M_{\text{fin}}$
and $M_{\text{fin}}'$ be two free factors of $M$ containing $R_{Q}$
of respective finite ranks $t$ and $t'$. Let $F\in\text{Mat}_{k\times t}\left(\mathcal{A}\right)$
and $F'\in\text{Mat}_{k\times t'}\left(\mathcal{A}\right)$ be matrices
whose columns form respective bases of $M_{\text{fin}}$ and $M_{\text{fin}}'$.
The intersection $M_{\text{fin}}\cap M_{\text{fin}}'$ satisfies,
too, the assumptions in Proposition \ref{prop: duality is well-defined before its definition}:
it contains $R_{Q}$, it is a free factor of $M$ by Corollary \ref{Finite intersection of free factors is a free factor},
and it is finitely generated by Proposition \ref{Intersection preserves freeness of extension}.
Thus, it is enough to show that each of $M_{\text{fin}}$ and $M_{\text{fin}}'$
gives rise to the same left submodule as $M_{\text{fin}}\cap M_{\text{fin}}'$.
Since this intersection is a free factor of both $M_{\text{fin}}$
and $M_{\text{fin}}'$ (again, by Corollary \ref{Finite intersection of free factors is a free factor}),
we may assume without loss of generality that $M_{\text{fin}}'\leq_{*}M_{\text{fin}}$.
By the non-dependence on the chosen basis, we may further assume that
$F'$ is the submatrix of $F$ obtained by keeping the first $t'$
columns. Let $B'\in\text{Mat}_{t'\times m}\left(\mathcal{A}\right)$
be the unique matrix such that $Q=F'B'$. Let $B\in\text{Mat}_{t\times m}\left(\mathcal{A}\right)$
be obtained from $B'$ by adding $t-t'$ rows of zeros. Then $FB=F'B'=Q$,
so $B$ satisfies (\ref{eq: duality definition for R_Q}). Adding
trivial rows clearly does not alter the row space, and so $L_{B}=L_{B'}$
as needed.

Since $Q$ is a left-multiple of the matrix $B$ in (\ref{eq: duality definition for R_Q}),
by Claim \ref{matrix product Column space contained} we have $L_{Q}\leq L_{B}$.
Finally, the row space $L_{B}$ of $B$ is generated by its $t$ rows
and is therefore finitely generated. Since $t=\text{rk}M_{\text{fin}}$
and $M_{\text{fin}}$ is a free factor of $M$ we deduce $\text{rk}L_{B}\leq\text{rk}M$.
\end{proof}
\begin{example}
\label{exa:Q-dual of everything}The $Q$-dual of $\mathcal{A}^{k}$
is $L_{Q}$ and the $Q$-dual of $\mathcal{A}^{m}$ is $R_{Q}$. Indeed,
expressing $Q=I\cdot Q$, where the columns of $I\in\text{Mat}{}_{k\times k}\left(\mathcal{A}\right)$
form the standard basis of $\mathcal{A}^{k}$, the $Q$-dual of $\mathcal{A}^{k}$
is $L_{Q}$.
\end{example}
\begin{example}
For an arbitrary field $K$  and $F=\left\langle x\right\rangle $
the free group on a single generator $x$, let $Q$ be the matrix
with the single entry $x^{3}-1$. Let $R$ be the (right) ideal generated
by $x-1$. Then $x-1$ is a basis for $R$, and we may express $Q$
as $x^{3}-1=\left(x-1\right)\left(x^{2}+x+1\right).$ Thus, the $Q$-dual
of $R$ is the (left) ideal of $K\left[F\right]$ generated by $x^{2}+x+1$.
\end{example}
\begin{rem}
\label{rem: homological language}The $Q$-duality introduced in Definition
\ref{Def: Q-dual} is a submodule-based variant of the duality $V\mapsto V^{\circ}:=\text{Ext}_{1}^{\mathcal{A}}\left(V,\mathcal{A}\right)$
developed by Cohn \cite[Chap.~5.2]{Cohn2006}, between finitely presented
algebraic left and right modules over a FIR. The term \emph{algebraic}
here (also called \emph{bound} by Cohn) refers to modules having no
non-trivial free quotients, or, equivalently, modules that arise as
the quotient of an algebraic extension. Our choice of stating Definition
\ref{Def: Q-dual} and the subsequent results in terms of submodules
and matrices, rather than homological algebra, serves our algorithmic
purposes and ensures the results remain accessible to a broader audience
who may take interest in the implementation of such algorithms. It
is worthwhile, however, to recast $Q$-duality in the language of
homological algebra, where it arises as a variant of the following
duality. We omit the proofs of the assertions in the following two
paragraphs, as the analogous results on $Q$-duality will be presented
later in this section. 

Consider a finitely presented right $\mathcal{A}$-module $V$. Then
$V^{\circ}=\text{Ext}_{1}^{\mathcal{A}}\left(V,\mathcal{A}\right)$
is naturally endowed with a left $\mathcal{A}$-module structure.
Moreover, one may show that $V^{\circ}$ is always finitely presented
and algebraic, and if $V$ itself is algebraic, then $\left(V^{\circ}\right)^{\circ}\cong V$.
In the general case where $V$ is not necessarily algebraic, $\left(V^{\circ}\right)^{\circ}$
is a direct summand of $V$ giving rise to a free quotient $V/(V^{\circ})^{\circ}$.

Suppose now that $V$ is an algebraic right $\mathcal{A}$-module
having the free resolution $0\rightarrow\mathcal{A}^{m}\overset{\phi}{\rightarrow}\mathcal{A}^{k}\rightarrow V\rightarrow0$.
Let $Q\in\text{Mat}_{k\times m}\left(\mathcal{A}\right)$ be the matrix
representing the map $\phi$ via $v\mapsto Qv$, so that $V\cong\mathcal{A}^{k}/R_{Q}$.
Then the same matrix $Q$ also represents, via $\psi\mapsto\psi Q$,
the induced left $\mathcal{A}$-module homomorphism $\phi^{*}:\text{Hom}\left(\mathcal{A}^{k},\mathcal{A}\right)\rightarrow\text{Hom}\left(\mathcal{A}^{m},\mathcal{A}\right)$
of pre-composition with $\phi$. It follows that $V^{\circ}\cong\mathcal{A}^{m}/L_{Q}$.
The map $U\mapsto\left(V/U\right)^{\circ}$ then defines a duality
between algebraic submodules of $V\cong\mathcal{A}^{k}/R_{Q}$ and
algebraic submodules of $V^{\circ}\cong\mathcal{A}^{m}/L_{Q}$. The
$Q$-duality is a variant of this duality, defined for arbitrary matrices
$Q$ (whose action by left multiplication is not necessarily injective),
and formulated at the level of intermediate free submodules of the
extensions $R_{Q}\leq\mathcal{A}^{k}$ and $L_{Q}\leq\mathcal{A}^{m}$,
rather than submodules of the quotient modules $\mathcal{A}^{k}/R_{Q}$
and $\mathcal{A}^{m}/L_{Q}$.
\end{rem}

\subsection{Properties of the Duality Induced by a Matrix}

We next aim to show that the mappings $M\mapsto M^{*Q}$ previously
defined form a duality between the respective algebraic extensions
of the column and row spaces of $Q$.
\begin{prop}[\textbf{The $Q$-dual is invariant under free extensions}]
\textbf{}\label{prop:Q-dual is invariant under free extensions}
Let $M_{1},M_{2}\in[R_{Q},\infty)$ or $M_{1},M_{2}\in[L_{Q},\infty)$.
If $M_{1}\leq_{*}M_{2}$ then $M_{1}^{*Q}=M_{2}^{*Q}$. 
\end{prop}
\begin{proof}
Since $M_{1}\leq_{*}M_{2}$, any free factor of $M_{1}$ is also a
free factor of $M_{2}$. Thus, the same $M_{\text{fin}}$ may be used
in defining both $Q$-duals. 
\end{proof}
\begin{prop}
Taking the $Q$-dual inverts the order of inclusions.
\end{prop}
\begin{proof}
Let $M_{1},M_{2}\in[R_{Q},\infty)$ such that $M_{1}\leq M_{2}.$
We first show that we may choose for $i\in\left\{ 1,2\right\} $ finitely
generated free factors $M_{i,\text{fin}}\leq_{*}M_{i}$ containing
$R_{Q}$ such that $M_{1,\text{fin}}\leq M_{2,\text{fin}}$. Choose
first a finitely generated free factor $M_{2,\text{fin}}$ of $M_{2}$
containing $R_{Q}$. The intersection $M_{1}\cap M_{2,\text{fin}}$
contains $R_{Q}$ and is a free factor of $M_{1}\cap M_{2}=M_{1}$
by Proposition \ref{Intersection preserves freeness of extension}.
Choose a finitely generated free factor $M_{1,\text{fin}}$ of $M_{1}\cap M_{2,\text{fin}}$
containing $R_{Q}$. By definition, $M_{1,\text{fin}}\leq M_{2,\text{fin}}$
and by the transitivity of free factors $M_{1,\text{fin}}\leq_{*}M_{1}$.\\
For each $i\in\left\{ 1,2\right\} $ write $Q=F_{i}B_{i}$ where $F_{i}$
is a matrix whose columns form a basis for $M_{i,\text{fin}}$. Since
$M_{1,\text{fin}}\leq M_{2,\text{fin}}$, there exists a matrix $C$
such that $F_{1}=F_{2}C$, so $F_{2}B_{2}=Q=F_{1}B_{1}=F_{2}CB_{1}.$
Since the columns of $F_{2}$ form a basis, we deduce that $B_{2}=CB_{1}$,
so 
\[
M_{2}^{*Q}=L_{B_{2}}=L_{CB_{1}}\subseteq L_{B_{1}}=M_{1}^{*Q}.\qedhere
\]
\end{proof}
\begin{prop}
\label{Double dual is free factor}For every $M\in[R_{Q},\infty)$
or $M\in[L_{Q},\infty)$, the double dual $\left(M^{*Q}\right)^{*Q}$
is a free factor of $M$.
\end{prop}
\begin{proof}
Let $M\in[R_{Q},\infty)$. By possibly replacing $M$ with a finitely
generated free factor containing $R_{Q}$, we may suppose without
loss of generality that $t:=\text{rk}M<\infty$. We first set up the
matrices defining the dual and the double dual. Let $F\in\text{Mat}_{k\times t}\left(\mathcal{A}\right)$
be a matrix whose columns form a basis for $M$. Let $B\in\text{Mat}_{t\times m}\left(\mathcal{A}\right)$
be the matrix satisfying (\ref{eq: duality definition for R_Q}) and
denote the rank of its row space $L_{B}$ by $s$. Let $B'\in\text{Mat}_{s\times m}\left(\mathcal{A}\right)$
be a matrix whose rows form a basis for $L_{B}$, and let $F'\in\text{Mat}_{k\times s}\left(\mathcal{A}\right)$
be the matrix satisfying $Q=F'B'$ as in (\ref{eq: duality definition for L_Q}).
By definition, $M^{*Q}=L_{B}$ and $\left(M^{*Q}\right)^{*Q}=R_{F'}$.

We next show that $\left(M^{*Q}\right)^{*Q}$ is contained in $M$.
Since $L_{B}=L_{B'}$, there exist matrices $C\in\text{Mat}_{s\times t}\left(\mathcal{A}\right)$
and $C'\in\text{Mat}_{t\times s}\left(\mathcal{A}\right)$ such that
$B'=CB$ and $B=C'B'.$ Thus, $F'B'=Q=FB=FC'B'$. Since the rows of
$B'$ form a basis we deduce that $F'=FC'$. In particular, 
\[
\left(M^{*Q}\right)^{*Q}=R_{F'}\leq R_{F}=M.
\]

It remains to show that the extension $R_{F'}\leq R_{F}$ is free,
which we show by following the outline of the argument of Umirbaev
in \cite{Umirbaev1996}. By Proposition \ref{prop:equivalent characterizations of free factor},
it is enough to construct an $\mathcal{A}$-module retraction $\tau:R_{F}\rightarrow R_{F'}$.
Since $B'=CB=CC'B'$ and the rows of $B'$ form a basis, $CC'$ is
the $s\times s$ identity matrix. For every $u\in R_{F}$ denote by
$\left[u\right]_{F}\in\text{Mat}_{t\times1}\left(\mathcal{A}\right)$
the coordinate vector of $u$ in the basis which is the columns of
$F$, i.e., such that $u=F\left[u\right]_{F}$. Let $\tau:M\rightarrow M$
be defined by $\tau\left(u\right)=F'C\left[u\right]_{F}.$ The map
$\tau$ is a homomorphism of right $\mathcal{A}$-modules whose image
clearly lies in $R_{F'}$. Let $v=F'\left[v\right]_{F'}\in R_{F'}$.
Since $F'=FC',$ we have $\left[v\right]_{F}=C'\left[v\right]_{F'}$
so 
\[
\tau\left(v\right)=F'CC'\left[v\right]_{F'}=F'\left[v\right]_{F'}=v,
\]
and $\tau\mid_{R_{F'}}=Id_{R_{F'}}$.
\end{proof}
We are now ready to prove Theorem \ref{Thm: Equivalent conditions of algebraic extensions via duality in Introduction},
which characterizes algebraicity via duality. Recall, specifically,
that this theorem states that if $M\in[R_{Q},\infty)$, then $(1)$
$M$ is algebraic over $R_{Q}$ if and only if $(2)$ $\left(M^{*Q}\right)^{*Q}=M$
if and only if $(3)$ $\text{rk}\left(M^{*Q}\right)=\text{rk}M$ (with
an equivalent statement if $M\in[L_{Q},\infty)$).
\begin{proof}[Proof of Theorem \ref{Thm: Equivalent conditions of algebraic extensions via duality in Introduction}]
 Let $M\in[R_{Q},\infty)$.\\
\uline{(1=>2)}: Suppose that $R_{Q}\leq_{\text{alg}}M$. By Proposition
\ref{Double dual is free factor}, we have $R_{Q}\leq\left(M^{*Q}\right)^{*Q}\leq_{*}M$,
and algebraicity implies that $\left(M^{*Q}\right)^{*Q}=M$.\\
\uline{(2=>3)}: Suppose that $\left(M^{*Q}\right)^{*Q}=M$. Since
taking the $Q$-dual cannot increase rank by Proposition \ref{prop: duality is well-defined before its definition},
we get:
\[
\text{rk}M=\text{rk}\left(\left(M^{*Q}\right)^{*Q}\right)\leq\text{rk}M^{*Q}\leq\text{rk}M,
\]
and so these inequalities hold as equalities and $\text{rk}M^{*Q}=\text{rk}M$.\\
\uline{(3=>1)}: Suppose that $\text{rk}M^{*Q}=\text{rk}M$, and
note that this is an equality of finite numbers as the rank of a $Q$-dual
is finite. Let $M'$ be some intermediate free factor $R_{Q}\leq M'\leq_{*}M$.
By Proposition \ref{prop:Q-dual is invariant under free extensions},
the $Q$-duals of $M'$ and $M$ are equal, so
\[
\text{rk}M'\leq\text{rk}M=\text{rk}M^{*Q}=\text{rk}\left(M'\right)^{*Q}\leq\text{rk}M'.
\]
Thus, $\text{rk}M'=\text{rk}M$ from which the freeness of $M'\leq_{*}M$
implies $M'=M.$
\end{proof}
\begin{prop}
\label{Prop: Q-dual is algebraic}For every $M\in[R_{Q},\infty)$
we have $L_{Q}\leq_{\text{alg}}M^{*Q}$ and for every $M\in[L_{Q},\infty)$
we have $R_{Q}\leq_{\text{alg}}M^{*Q}.$
\end{prop}
\begin{proof}
Let $M\in[R_{Q},\infty)$. By Proposition \ref{Double dual is free factor},
the double dual $\left(M^{*Q}\right)^{*Q}$ is a free factor of $M$.
Since the $Q$-dual is invariant under a free extension (Proposition
\ref{prop:Q-dual is invariant under free extensions}), we obtain
an equality of the triple and single $Q$-duals, namely $\left(\left(M^{*Q}\right)^{*Q}\right)^{*Q}=M^{*Q}.$
The claim now follows by applying Theorem \ref{Thm: Equivalent conditions of algebraic extensions via duality in Introduction}
to $M^{*Q}$.
\end{proof}
Several important corollaries of the above now follow. The first is
that when restricted to algebraic extensions, the $Q$-duality is
a bijection:
\begin{cor}
The $Q$-duality constitutes a duality between algebraic extensions
of the column and row spaces of $Q$ which preserves rank and inverts
the order of inclusions.
\end{cor}
Recalling Theorem \ref{Theorem: existence and uniqueness of Algebraic - free decomposition},
the row and column spaces of $Q$ are finitely generated and so for
any extension $M$ of one of them the algebraic closure is guaranteed
to exist. The second corollary we now draw is the main result of this
section -- Theorem \ref{Theorem: algebraic-free decomposition-1}
-- which states that the algebraic closure of $M$ is obtained as
the double $Q$-dual $\left(M^{*Q}\right)^{*Q}$:
\begin{proof}[Proof of Theorem \ref{Theorem: algebraic-free decomposition-1}]
 This is an immediate corollary from Propositions \ref{Prop: Q-dual is algebraic}
and \ref{Double dual is free factor}.
\end{proof}
\begin{rem}
This result recovers, at the level of submodules of free modules,
the classical reflexivity $\left(V^{\circ}\right)^{\circ}\cong V$
for finitely presented algebraic modules over a free ideal ring, discussed
in Remark \ref{rem: homological language}.
\end{rem}
The next corollary rephrases Theorem \ref{Theorem: algebraic-free decomposition-1}
for the case that one is only interested in testing for algebraicity
or freeness of an extension:
\begin{cor}[\textbf{Criteria for freeness and algebraicity via duality}]
\label{Cor: freeness testing via duality} Let $M\in[M_{0},\infty)$
where $M_{0}$ is either $R_{Q}$ or $L_{Q}$.
\begin{enumerate}
\item $M_{0}\leq_{*}M$ if and only if $\left(M^{*Q}\right)^{*Q}=M_{0}$.
\item $M_{0}\leq_{\text{alg}}M$ if and only if $\left(M^{*Q}\right)^{*Q}=M$.
\end{enumerate}
\end{cor}
We finish this section by calculating the $Q$-duals of $R_{Q}$ and
$L_{Q}$ themselves:
\begin{prop}
\label{Q-dual of column space}We have $R_{Q}\leq_{\text{alg}}L_{Q}^{*Q}\leq_{*}\mathcal{A}^{k}$
and $L_{Q}\leq_{\text{alg}}R_{Q}^{*Q}\leq_{*}\mathcal{A}^{m}$, namely,
the $Q$-dual of the column space of $Q$ is the algebraic closure
of the row extension $L_{Q}\leq\mathcal{A}^{m}$ and vice versa.
\end{prop}
\begin{proof}
By Example \ref{exa:Q-dual of everything}, $\left(\mathcal{A}^{k}\right)^{*Q}=L_{Q}$
so $L_{Q}^{*Q}$ is the double $Q$-dual of $\mathcal{A}^{k}$, and
is thus the algebraic closure of the extension $R_{Q}\leq\mathcal{A}^{k}$.
\end{proof}
We note an elegant corollary of the above, also appearing in Cohn's
book \cite[Prop.~5.4.2]{Cohn2006}:
\begin{cor}
\label{cor: nice corollary about ranks}Let $Q\in\text{Mat}{}_{k\times m}\left(\mathcal{A}\right)$\textup{
and let $\ensuremath{R_{Q}\leq_{\text{alg}}R_{\text{max}}\leq_{*}\mathcal{A}^{k}}$
and $\ensuremath{L_{Q}\leq_{\text{alg}}L_{\text{max}}\leq_{*}\mathcal{A}^{m}}$}
be the associated algebraic closures of the extensions of its column
space and row spaces, respectively, to the free modules in which they
lie. Then:
\[
\text{rk}R_{Q}=\text{rk}L_{\text{max}},\ \text{rk}L_{Q}=\text{rk}R_{\text{max}}.
\]
\end{cor}
\begin{rem}
When $\mathcal{A}$ is a field, or more generally, a division ring,
every $\mathcal{A}$-module is free, namely, admits a basis (the well-known
proofs for vector spaces extend to modules over division rings). In
this case, it follows from Proposition \ref{prop:equivalent characterizations of free factor}
that every $\mathcal{A}$-module extension is free, so the algebraic
closures in Corollary \ref{cor: nice corollary about ranks} are $R_{\text{max}}=R_{Q}$
and $L_{\text{max}}=L_{Q}$. Corollary \ref{cor: nice corollary about ranks}
then states the familiar result that the column rank and row rank
of a given matrix over a field (or a division ring) are equal. For
a matrix over an arbitrary free ideal ring this is not necessarily
the case: see Example \ref{exa:row rank not equal to column rank}.
\end{rem}
\begin{example}
\label{exa:row rank not equal to column rank}Let $F=\left\langle x,y\right\rangle $
be the free group on $2$ generators and let $K$ be an arbitrary
field. Let $\mathcal{A}:=K\left[F\right]$ be the free group algebra
and $Q=\left(x-1,x-1,y-1\right)\in\text{Mat}_{1\times3}\left(\mathcal{A}\right)$.
We first observe that the row and column ranks of $Q$ disagree. Clearly,
the row space of $Q$ is of rank $\text{rk}L_{Q}=$$1$. Its column
space $R_{Q}=\left(x-1\right)\mathcal{A}+\left(y-1\right)\mathcal{A}$
is the augmentation ideal of $\mathcal{A}$ (see Section \ref{sec: algs for free groups}).
The elements $x-1$ and $y-1$ form a basis for $R_{Q}$, so $\text{rk}R_{Q}=2\neq\text{rk}L_{Q}$.
We next verify that the equalities $\text{rk}R_{\text{max}}=\text{rk}L_{Q}$
and $\text{rk}L_{\text{max}}=\text{rk}R_{Q}$ of Corollary \ref{cor: nice corollary about ranks}
are satisfied. The algebraic closure $R_{\text{max}}$ of $R_{Q}$
in $\mathcal{A}$ is a free factor of $\mathcal{A}$, but since $\text{rk}\mathcal{A}=1$,
its only free factors are itself and the trivial submodule. Thus,
$R_{\text{max}}=\mathcal{A}$ and $\text{rk}R_{\text{max}}=1=\text{rk}L_{Q}$.
The algebraic closure $L_{\text{max}}$ of $L_{Q}$ in $\mathcal{A}^{3}$
is the $Q$-dual of $R_{Q}$ by Proposition \ref{Q-dual of column space}.
Writing $Q=\left(x-1,y-1\right)\cdot\begin{pmatrix}1 & 1 & 0\\
0 & 0 & 1
\end{pmatrix},$ we deduce that $L_{\text{max}}=\mathcal{A}\left(1,1,0\right)+\mathcal{A}\left(0,0,1\right)$
and hence $\text{rk}L_{\text{max}}=2=\text{rk}R_{Q}$.
\end{example}

\section{\label{sec:The-Duality in the free group algebra}The Duality Induced
by $w-1$ in the Free Group Algebra}

In this section we specialize the duality of the previous section
to the case where $\mathcal{A}$ is a free group algebra. Our objective
is to rephrase the $Q$-duality in a way which uses only \uline{right}
$\mathcal{A}$-modules. This is possible using an explicit isomorphism
between $\mathcal{A}$ and its opposite ring. In the specific case
where $w$ is a word in the free group and $Q$ is the $\dot{1\times1}$
matrix $(w-1)$, the $Q$-duality associates to each extension inside
${\cal A}$ of the right ${\cal A}$-module $(w-1){\cal A}$, an object
of the same type: an (algebraic) extension of the same right $\mathcal{A}$-module
$\left(w-1\right){\cal A}$.

\subsection{$Q$-duality in the Free Group Algebra}

Fixing a field $K$ and a free group $F$, we continue with the free
ideal ring $\mathcal{A}=K\left[F\right]$.
\begin{defn}[right-left inversion]
\label{def: iota} Let $\iota:\mathcal{A}\rightarrow\mathcal{A}$
be the (unique) $K$-linear extension of the inversion map $u\mapsto u^{-1}$
on $F$, given explicitly by $\iota\left(\sum_{u\in F}\alpha_{u}u\right)=\sum_{u\in F}\alpha_{u}u^{-1}.$
Extend $\iota$ to matrices over $\mathcal{A}$ by letting $\iota\left(Q\right)$
be the element-wise application of $\iota$ to the transpose $Q^{T}$
of $Q$. 
\end{defn}
Observe that for all $f,g\in\mathcal{A}$ we have $\iota\left(fg\right)=\iota\left(g\right)\iota\left(f\right)$.
Thus, $\iota$ is a ring anti-automorphism of $\mathcal{A}$, and
equivalently a ring isomorphism $\iota:\mathcal{A}^{op}\rightarrow\mathcal{A}$.
Consequently, $\iota$ interchanges left and right structures. In
particular, if $J$ is a left ideal of $\mathcal{A}$, then $J$ is
a right ideal of $\mathcal{A}^{op}$, and its image $\iota\left(J\right)$
is a right ideal of $\mathcal{A}$. Since $\iota$ is an isomorphism,
ranks are preserved. 

More generally, for all matrices $A$ and $B$ over $\mathcal{A}$
such that the product $AB$ is defined, the transpose in Definition
\ref{def: iota} ensures that the extended $\iota$ satisfies $\iota\left(AB\right)=\iota\left(B\right)\iota\left(A\right)$.
Thus, for a matrix $Q$ over $\mathcal{A}$, column extensions of
$Q$ are mapped by $\iota$ to row extensions of $\iota\left(Q\right)$,
and vice versa. This correspondence sends bases to bases while preserving
containment, freeness and algebraicity. Applying $Q$-duality and
then inverting with $\iota$ is the same as inverting and then applying
$\iota\left(Q\right)$-duality, namely, $\iota\left(M^{*Q}\right)=\left(\iota\left(M\right)\right){}^{*\iota\left(Q\right)}$.
Thus, we may specialize the results of the previous section to the
case of free group algebras using right $\mathcal{A}$-modules alone:
\begin{thm}[\textbf{$Q$-Duality in the free group algebra}]
\label{Thm: Q-duality in the Free Group Algebra} Let $Q$ be a matrix
over the free group algebra $K\left[F\right]$. The map $M\mapsto\iota\left(M^{*Q}\right)$
constitutes a duality between the algebraic extensions of the column
space of $Q$ and those of $\iota\left(Q\right)$, with inverse $M\mapsto\iota\left(M^{*\iota\left(Q\right)}\right).$
The duality preserves rank and inverts the order of inclusion. Taking
the double dual of a (not necessarily algebraic) column extension
$M$ of $Q$, first with respect to $Q$ and then with respect to
$\iota\left(Q\right)$, returns the unique algebraic extension of
$R_{Q}$ which is a free factor of $M$: 
\[
R_{Q}\leq_{\text{alg}}\iota\left(\left(\iota\left(M^{*Q}\right)\right)^{*\iota\left(Q\right)}\right)\leq_{*}M.
\]
\end{thm}

\subsection{The Duality of Algebraic Extensions of $\left(w-\lambda\right)\mathcal{A}$}

We further specialize to the case where $w\in F$ is a fixed word,
$\lambda\in K^{*}$ an invertible scalar and $Q_{\lambda}$ the $1\times1$
matrix containing the element $w-\lambda$. In this case, the $1\times1$
matrix $\iota\left(Q_{\lambda}\right)$ contains the entry $\iota(w-\lambda)=w^{-1}-\lambda$,
which generates the same right ideal in $\mathcal{A}$ as $Q_{\lambda^{-1}}=w-\lambda^{-1}$:
indeed, $w-\lambda^{-1}$ is obtained from $w^{-1}-\lambda$ by right
multiplication with the invertible element $-\lambda^{-1}w$. Furthermore,
taking the duals with respect to both matrices yields the same result:
\begin{lem}
\label{lem: duality with Q of inverse instead of iota of Q}Let $w\in F$
and $\lambda\in K^{*}$. If $w-\lambda^{-1}\in M\leq\mathcal{A}$
then $M^{*\iota\left(Q_{\lambda}\right)}=M^{*Q_{\lambda^{-1}}}.$
\end{lem}
\begin{proof}
As $(w-\lambda^{-1}){\cal A}=\iota(w-\lambda){\cal A}$, the module
$M$ extends $R_{\iota(Q)}$ as well. Let $M_{\text{fin}}$ be a finitely
generated free factor of $M$ containing $R_{\iota\left(Q\right)}$.
Let $f$ be a row vector whose elements form a basis for $M_{\text{fin}}$,
and let $g$ be a column vector such that 
\begin{equation}
w-\lambda^{-1}=fg.\label{eq: coeffs}
\end{equation}
Multiplying (\ref{eq: coeffs}) from the right by $-\lambda w^{-1}$
expresses $w^{-1}-\lambda=fg\cdot\left(-\lambda w^{-1}\right)$, so
$M^{*Q_{\lambda^{-1}}}$ and $M^{*\iota\left(Q_{\lambda}\right)}$
are the left ideals generated by the entries of $g$ and $gw^{-1}$,
respectively. (The coefficient $-\lambda$ can be omitted as it clearly
does not change the ideal generated: $-\lambda$ commutes with the
entries of $g$, while $w^{-1}$ generally does not.) To show that
the two ideals coincide, multiply (\ref{eq: coeffs}) from the left
by $\left(-\lambda gw^{-1}\right)$ to obtain
\[
-\lambda g+gw^{-1}=-\lambda gw^{-1}fg\ \ \ \Longrightarrow\ \ \ gw^{-1}=\left(\lambda I-\lambda gw^{-1}f\right)g,
\]
where $I$ is the identity matrix. This shows that $M^{*\iota\left(Q_{\lambda}\right)}\subseteq M^{*Q_{\lambda^{-1}}}$.
A similar argument, multiplying (\ref{eq: coeffs}) from the left
by $g$ and from the right by $w^{-1}$, shows that $g=\left(\lambda^{-1}I+gf\right)gw^{-1}$,
and so $M^{*Q_{\lambda^{-1}}}\subseteq M^{*\iota\left(Q_{\lambda}\right)}$
as well.
\end{proof}
Lemma \ref{lem: duality with Q of inverse instead of iota of Q} allows
us to use the duality with respect to $Q_{\lambda^{-1}}$ instead
of $\iota\left(Q_{\lambda}\right)$, obtaining a more symmetric form
for the duality in this case:
\begin{cor}[\textbf{$Q$-duality for algebraic extensions of $w-\lambda$}]
\textbf{}\label{cor: w-duality} For every $w\in F$ and $\lambda\in K^{*}$,
the maps $J\mapsto\iota\left(J^{*Q_{\lambda}}\right)$ and $J\mapsto\iota\left(J^{*Q_{\lambda^{-1}}}\right)$
constitute a duality between the respective algebraic extensions of
the right ideals $\left(w-\lambda\right)\mathcal{A}$ and $\left(w-\lambda^{-1}\right)\mathcal{A}$
in $\mathcal{A}$. The duality preserves rank and inverts the order
of inclusion. For a (not necessarily algebraic) extension $\left(w-\lambda\right)\mathcal{A}\leq J$
we have 
\[
\left(w-\lambda\right)\mathcal{A}\leq_{\text{alg}}\iota\left(\left(\iota\left(J^{*Q_{\lambda}}\right)\right)^{*Q_{\lambda^{-1}}}\right)\leq_{*}J.
\]
\end{cor}
When $\lambda=1$ (or $\lambda=-1$) we have $\lambda=\lambda^{-1}$,
in which case the map $J\mapsto\iota\left(J^{*Q_{^{\lambda}}}\right)$
is an involution on the set of algebraic extensions of the ideal $\left(w-1\right)\mathcal{A}$
(or $\left(w+1\right)\mathcal{A}$), proving Theorem \ref{thm: there exists an involution on extensions of w-1}
as a special case of Corollary \ref{cor: w-duality}.

\section{\label{sec: Theorem about Phi and Duality}Proof of Theorem \ref{Thm: Phi invariant under duality}}

In the current section, it is imperative that the field $K$ be finite.
So fix a finite field $K$ of order $q$ and a word $w\in F$, and
let $\mathcal{A}=K\left[F\right]$. Recall from Definition \ref{def: Big phi}
that for an extension $I\leq J$ of free right $\mathcal{A}$-modules
with $\text{rk}J<\infty$, the quantity $\phi_{I,J}\left(N\right)$
is defined to be $q^{N}\mathbb{P}\left(I\leq\ker\varphi_{J}\right)$,
where the probability is taken over an independently and uniformly
chosen right $\mathcal{A}$-module structure on $K^{N}$ and $\mathcal{A}$-module
homomorphism $\varphi_{J}:J\rightarrow K^{N}$.

To see that there are finitely many right $\mathcal{A}$-module structures
on $K^{N}$, recall that $r$ marks the rank of $F$ and fix a basis
$B=\left\{ b_{1},\ldots,b_{r}\right\} $ of $F$. Every function $g\colon B\to GL_{N}\left(K\right)$
extends uniquely to a morphism of $K$-algebras $\hat{g}\colon\mathcal{A}\to\text{Mat}_{N\times N}\left(K\right)$.
This defines a right ${\cal A}$-module $M_{g}^{\text{right}}$ (respectively,
a left $\mathcal{A}$-module $M_{g}^{\text{left}}$) on the $K$-vector
space $K^{N}$ of row vectors (respectively, column vectors) of length
$N$ with entries in $K$. This recipe gives a bijection between all
possible maps $g$ as above and all right (or left) ${\cal A}$-module
structures on $K^{N}$. Thus, there are exactly $\left|GL_{N}\left(K\right)\right|^{r}$
right $\mathcal{A}$-module structures on $K^{N}$ -- a finite amount
since $K$ is finite.

For every $g:B\to GL_{N}\left(K\right)$, an $\mathcal{A}$-module
homomorphism $J\rightarrow M_{g}^{\text{right}}$ is uniquely determined
by the images in $K^{N}$ of some fixed basis of $J$. Since $J$
is free, each such choice of $\text{rk}J$ images in $K^{N}$ extends
to an $\mathcal{A}$-module homomorphism $J\rightarrow M_{g}^{\text{right}}.$
Thus, there are $\left|K^{N}\right|^{\text{rk}J}$ possible $\mathcal{A}$-module
homomorphisms $\varphi_{J}:J\rightarrow M_{g}^{\text{right}}$ --
again, finitely many as $K$ is finite and $\text{rk}J<\infty$. 

We now turn to the proof of Theorem \ref{Thm: Phi invariant under duality},
which states that $\phi$ is invariant under $Q$-duality when $I$
and $J$ are algebraic extensions of the column space of the matrix
$Q\in\text{Mat}_{k\times m}\left(\mathcal{A}\right)$. 
\begin{proof}[Proof of Theorem \ref{Thm: Phi invariant under duality}]
 Along the current proof we abbreviate $G_{N}:=\text{GL}_{N}(K)$.
We use the notation from the paragraphs above to write
\begin{align*}
\phi_{I,J}\left(N\right) & =q^{N}\cdot\frac{1}{\left|G_{N}\right|^{r}}\sum_{g\colon B\to G_{N}}\mathbb{P}_{\varphi_{J}:J\rightarrow M_{g}^{\text{right}}}\left(I\leq\ker\varphi_{J}\right).\\
\phi_{J^{*Q},I^{*Q}}\left(N\right) & =q^{N}\cdot\frac{1}{\left|G_{N}\right|^{r}}\sum_{g\colon B\to G_{N}}\mathbb{P}_{\varphi_{I^{*Q}}:I^{*Q}\rightarrow M_{g}^{\text{left}}}\left(J^{*Q}\leq\ker\varphi_{I^{*Q}}\right).
\end{align*}
The result will follow if we show that the two probabilities above
agree for each $g\colon B\to G_{N}$, and so we continue with a fixed
$g\colon B\to G_{N}$.

Let $k$ and $m$ be the respective number of rows and columns in
the matrix $Q$. Any algebraic extension $M$ of the column space
$R_{Q}$ is necessarily finitely generated. For example, this follows
by duality since $M=\left(M^{*Q}\right)^{*Q}$ so $M$ is finitely
generated as a $Q$-dual. In particular, $I$ and $J$ are finitely
generated, and denote their respective ranks by $s=\text{rk}I$ and
$t=\text{rk}J$. Choose matrices $S\in\text{Mat}_{k\times s}\left(\mathcal{A}\right)$
and $T\in\text{Mat}_{k\times t}\left(\mathcal{A}\right)$ whose columns
form bases for $I$ and $J$, respectively. Let $S'\in\text{Mat}_{s\times m}\left(\mathcal{A}\right)$
and $T'\in\text{Mat}_{t\times m}\left(\mathcal{A}\right)$ be the
unique matrices expressing $Q=SS'=TT'$. By Theorem \ref{Thm: Equivalent conditions of algebraic extensions via duality in Introduction}
and the algebraicity of $I$ and $J$, the rows of $S',T'$ are bases
(and not only generating sets) for $I^{*Q}$ and $J^{*Q}$, respectively.
Since $I\leq J$, there exists a $t\times s$ matrix $C=\left(c_{ij}\right)_{i,j}$
such that $S=TC$, so $TCS'=TT'$. This implies, as the columns of
$T$ are a basis, that $CS'=T'$.

Let $q_{1},q_{2},...,q_{t}$ denote the columns of the matrix $T$.
Since $q_{1},q_{2},...,q_{t}$ form a basis for $J$ as a free $\mathcal{A}$-module,
choosing the random homomorphism $\varphi_{J}:J\rightarrow M_{g}^{\text{right}}$
means choosing uniformly the images $v_{1},v_{2},...,v_{t}\in M_{g}^{\text{right}}$
of $q_{1},q_{2},...,q_{t}$. Let $p_{1},p_{2},...,p_{s}$ denote the
columns of $S$. The relation $S=TC$ expresses the $i$-th column
of $S$ using the columns of $T$ by the formula $p_{i}=\sum_{j=1}^{t}q_{j}c_{ji}$.
To similarly express the $\varphi_{J}$-image of $p_{i}$, recall
that the element $c_{ji}\in\mathcal{A}$ acts on the right $\mathcal{A}$-module
$M_{g}^{\text{right}}$ by right multiplication with the matrix $\widehat{g}\left(c_{ji}\right)\in\text{Mat}_{N\times N}\left(K\right)$,
where $\widehat{g}:\mathcal{A}\rightarrow\text{Mat}_{N\times N}\left(K\right)$
is the unique extension of the fixed $g$ to a $K$-algebra homomorphism.
We thus have $\varphi_{J}\left(p_{i}\right)=\sum_{j=1}^{t}v_{j}\hat{g}\left(c_{ji}\right)$.
In particular, the row vector $\left(\varphi_{J}\left(p_{1}\right),\varphi_{J}\left(p_{2}\right),...,\varphi_{J}\left(p_{s}\right)\right)$
of length $sN$ is obtained from the vector $\left(v_{1},v_{2},...,v_{t}\right)$
of length $tN$ by multiplication with the $K$-valued matrix $\hat{g}\left(C\right)$,
which is obtained from $C$ by substituting each $c_{ij}$ with the
$N\times N$ block $\hat{g}\left(c_{ij}\right)$. Thus, $I\leq\ker\varphi_{J}$
happens exactly when $\left(v_{1},v_{2},...,v_{t}\right)$ is in the
(left) kernel of $\hat{g}(C)$, which is of dimension $tN-\text{rk}\left(\hat{g}(C)\right)$
(recall that $\hat{g}(C)$ is a matrix over $K$ and its rank is the
usual notion of a rank of a matrix). Hence
\[
\mathbb{P}_{\varphi_{J}:J\rightarrow M_{g}^{\text{right}}}\left(I\leq\ker\varphi_{J}\right)=\frac{q^{tN-\text{rk}\left(\hat{g}(C)\right)}}{q^{tN}}=q^{-\text{rk}\left(\hat{g}(C)\right)}.
\]

Since $CS'=T',$ a similar calculation for $\varphi_{J^{*Q}}$ shows
that for a homomorphism $\varphi_{I^{*Q}}:I^{*Q}\rightarrow M_{g}^{\text{left}}$,
the images of the rows of $T'$ are obtained from the (column vector
of) images of the rows of $S'$ by multiplication, this time from
the left, with the same matrix $\hat{g}(C)$. Since its row and column
rank are equal, we similarly obtain 
\[
\mathbb{P}_{\varphi_{I^{*Q}}:I^{*Q}\rightarrow M_{g}^{\text{left}}}\left(J^{*Q}\leq\ker\varphi_{I^{*Q}}\right)=q^{-rk\left(\hat{g}(C)\right)},
\]
finishing the proof.
\end{proof}
The algebraicity of $I$ is essential to Theorem \ref{Thm: Phi invariant under duality}
since extending $I$ freely to $I\leq_{*}I'\leq J$ can result in
$\phi_{I',J}$ being smaller than $\phi_{I,J}$ while the $Q$-duals
remain the same. In contrast, the assumption that $R_{Q}\leq_{\text{alg}}J$
is, in fact, redundant. 
\begin{cor}
\label{Cor: Phi invariant more general}Let $Q$ be a matrix over
$\mathcal{A}$. Let $I$ and $J$ be extensions of the column space
of $Q$ such that $R_{Q}\leq_{\text{alg}}I\le J$ and $\text{rk}J<\infty$.
Then $\phi_{I,J}=\phi_{J^{*Q},I^{*Q}}.$
\end{cor}
\begin{proof}
Let $J'$ be the algebraic closure of the extension $R_{Q}\leq J$.
By Theorem \ref{Theorem: existence and uniqueness of Algebraic - free decomposition},
$J'$ contains $I$ since $R_{Q}\leq_{\text{alg}}I\leq J$. Since
$R_{Q}\leq_{\text{alg}}J'$, by Theorem \ref{Thm: Phi invariant under duality}
we have $\phi_{I,J'}\left(N\right)=\phi_{\left(J'\right)^{*Q},I^{*Q}}$$\left(N\right)$,
so it is enough to show that replacing $J'$ with $J$ leaves both
sides unchanged. Indeed, as $J'\leq_{*}J$, by Proposition \ref{prop:Q-dual is invariant under free extensions}
their $Q$-duals coincide and so $\phi_{J^{*Q},I^{*Q}}\left(N\right)=\phi_{\left(J'\right)^{*Q},I^{*Q}}\left(N\right)$.
Furthermore, a uniformly random homomorphism from $J$ can be selected
by first choosing the images of a basis of $J'$ and then, independently,
the images of elements extending this basis to a basis of $J$. Hence,
$\phi_{I,J}\left(N\right)=\phi_{I,J'}\left(N\right)$.
\end{proof}

\section{\label{sec: Algs for group algebras}An Algorithm for Computing the
Algebraic Closure in Free Group Algebras}

In this section we provide the algorithms of Theorem \ref{thm: There exists algorithms for alg-free decomp}
for finding the algebraic closure, testing for algebraicity and testing
for freeness (Algorithms \ref{alg: Get-Algebraic-Free}, \ref{alg: Is-Algebraic-Extension}
and \ref{alg: Is-Free-Extension} respectively) for extensions of
finitely generated free modules over the free group algebra. The algorithms
use the duality of Section \ref{sec:The-Duality-Induced}, together
with Rosenmann's algorithm from \cite{Rosenmann1993}. Let $\mathcal{A}$
denote again a free group algebra. Given an arbitrary finite set of
elements of $\mathcal{A}$, Rosenmann's algorithm finds a basis for
the right ideal generated by that set. This basis is deterministic
(depends, basically, only on the ideal and not on the particular generating
set), and we call it the \emph{Rosenmann Basis} of the ideal. See
Section \ref{subsec: Rosenmann's-Algorithm-for} below for more details
about Rosenmann's work.

We will need two minor extensions to Rosenmann's algorithm. First,
the duality involves the coefficients by which certain elements of
a module are expressed in terms of a given basis, so we describe in
Algorithm \ref{Algorithm: coefficient extraction} how to extract
the coefficients expressing an element of an ideal in its Rosenmann
basis. Second, we bootstrap Rosenmann's algorithm to also support
submodules of $\mathcal{A}^{k}$ for $k>1$.

\subsection{\label{subsec: Rosenmann's-Algorithm-for}Rosenmann's Algorithm}

In \cite{Rosenmann1993}, Rosenmann describes an algorithm which receives
a finite generating set for a right ideal $I\leq\mathcal{A}$ and
outputs a basis for $I$. The algorithm uses a well-order $<$ on
the elements of $F$ known as \emph{shortlex}, which orders words
by length, and breaks ties lexicographically (with respect to some
fixed arbitrary order on the set of words of length $1$ in $F$).
The shortlex order satisfies a form of invariance with respect to
suffixes: If $u,v\in F$ satisfy $u<v$ and $w\in F$ is a word such
that there is no cancellation in the product $vw$, then $uw<vw.$
Rosenmann's algorithm relies on shortlex by using length considerations,
as well as suffix-invarince, in both computations and arguments of
correctness. We note, however, that recent algorithms \cite{Seidel2025}
are able to preform similar tasks while applying a much wider class
of well-orders on $F$, without relying on suffix-invariance.

In more detail, with respect to the shortlex order, Rosenmann associates
to $I$ three distinguished objects of interest, each of which satisfies
some form of minimality. The first is a basis $\alpha_{1},\alpha_{2},...,\alpha_{m}$
for $I$ where $m=\text{rk}I\in\mathbb{Z}_{\geq0}$, which we call
the \emph{Rosenmann basis} of $I$. The second is a \emph{Schreier
transversal} $T_{I}$ for $I$: a subset of $F$ which is closed under
taking prefixes such that its $K$-linear span satisfies $Span_{K}\left(T_{I}\right)\cap I=\left\{ 0\right\} $
and $Span_{K}\left(T_{I}\right)+I=\mathcal{A}$. The third is a \emph{Gröbner
basis} for $I$, which is composed of the basis elements $\alpha_{1},\alpha_{2},...,\alpha_{m}$
together with elements $\beta_{1},\beta_{2},...,\beta_{m}$ of $I$
called their \emph{seconds}\footnote{In the very special case $I=\mathcal{A}$, the algorithm returns the
unity of $\mathcal{A}$ as $\alpha_{1}$ without a second: $\alpha_{1}$
alone is already a\emph{ }Gröbner basis for $I$ in this case.}. The algorithm outputs the Rosenmann basis and the Gröbner basis
of $I$. Notably, the output depends only on $I$ and not on the finite
generating set of $I$ given as input.

Since $T_{I}$ is a Schreier transversal, the $I$-coset of any element
$f\in\mathcal{A}$ contains a unique representative $\phi_{I}\left(f\right)$
lying in $Span_{K}\left(T_{I}\right)$. We call $\phi_{I}\left(f\right)$
the \emph{remainder of $f$ modulo the Rosenmann basis of $I$}. As
explained by Rosenmann, the Gröbner basis of $I$ allows for an algorithmic
reduction of an element $f\in\mathcal{A}$ to its remainder. Our goal
in this section is to explain how the Gröbner basis can also be used
to perform a long division with respect to the Rosenmann basis - that
is, to extract the unique coefficients $g_{1},g_{2},...,g_{m}\in\mathcal{A}$
expressing $f=\sum_{i=1}^{m}\alpha_{i}g_{i}$+$\phi_{I}\left(f\right)$
(Algorithm \ref{Algorithm: coefficient extraction}).

We summarize the reduction process and its key properties from \cite{Rosenmann1993}.
The well-ordering of $F$ extends to a well-order on finite subsets
of $F$ by declaring for distinct subsets $A,B\subseteq F$ that $A<B$
if the maximal element in their symmetric difference lies in $B$.
With respect to this order, for every $f\in\mathcal{A}$, the remainder
$\phi_{I}\left(f\right)$ is the unique element with minimal support
in its $I$-coset $f+I$. The reduction of $f$ to its remainder $\phi_{I}\left(f\right)$
is done by a sequence of reduction steps. Each step replaces the current
element with another element of the same coset $f+I$ with smaller
support. Since the order on supports is a well-order, the algorithm
must terminate after a finite number of steps.

We next explain the stopping condition. An element $f\in\mathcal{A}$
is reduced if and only if it lies in $Span_{K}\left(T_{I}\right)$.
Adopting Rosenmann's notation, the maximal element of $F$ in the
support of a nonzero element of $\mathcal{A}$ is called its \emph{LPP}
(``Leading Power Product''). The Schreier transversal $T_{I}$ is
characterized as the set of words of $F$ which do not have as a prefix
any of the $\text{LPP}$s of $\alpha_{1},\alpha_{2},..,\alpha_{m},\beta_{1},\beta_{2},...,\beta_{m}$.
Thus, using the Gröbner basis, we may determine whether a given word
$u\in F$ lies in $T_{I}$ by checking if it has any prefix among
this finite set.

For the reduction step, suppose that $f\in\mathcal{A}$ is not reduced.
Then there exists some $w\in F$ in the support of $f$ which has
as a prefix the LPP of some $\gamma\in\left\{ \alpha_{1},\alpha_{2},...,\alpha_{m},\beta_{1},\beta_{2},...,\beta_{m}\right\} $.
Write $w=\text{LPP}\left(\gamma\right)w'$, where $w'\in F$ and the
product is without cancellation. It follows from the suffix-invarariance
property of shortlex that $\text{LPP}\left(\gamma w'\right)=w$. Each
element in the Gröbner basis is \emph{monic}, meaning that the coefficient
of its $\text{LPP}$ is $1$. The reduction step now replaces $f$
by $f'=f-c\gamma w'$, where $c\in K^{*}$ is the coefficient of $w$
in $f$. Since $\gamma\in I$, the new element $f'$ lies in the same
$I$-coset as $f$. Furthermore, let $\Delta$ be the symmetric difference
of the supports of $f$ and $f'$. Then $\Delta$ is contained in
the support of $c\gamma w'$, whose $\text{LPP}$ is $w$. But $w\in\Delta$,
so the support of $f'$ is indeed smaller than the support of $f$.

Performing the reduction on some $f\in\mathcal{A}$ while keeping
track of the reduction steps, we obtain elements $s_{1},s_{2},...,s_{m},t_{1},t_{2},...,t_{m}\in\mathcal{A}$
such that $f=\sum_{i=1}^{m}\alpha_{i}s_{i}+\sum_{i=1}^{m}\beta_{i}t_{i}+\phi_{I}\left(f\right)$.
The seconds $\beta_{1},\beta_{2},...,\beta_{m}$ are obtained during
Rosenmann's algorithm from the firsts as some explicit $\mathcal{A}$-linear
combination $\left(\beta_{1},\beta_{2},...,\beta_{m}\right)=\left(\alpha_{1},\alpha_{2},...,\alpha_{m}\right)C$
where $C\in\text{Mat}{}_{m\times m}\left(\mathcal{A}\right)$ is an
invertible upper-triangular matrix. We use this matrix to rewrite
$f=\sum_{i=1}^{m}\alpha_{i}g_{i}+\phi_{I}\left(f\right)$, where $g_{1},g_{2},...,g_{m}$
are defined by $g=s+Ct$ for $s$ and $t$ the column vectors in $\mathcal{A}^{m}$
with respective entries $s_{1},s_{2},...,s_{m}$ and $t_{1},t_{2},...,t_{m}$.
While not explicitly done so by Rosenmann, we consider the matrix
$C$ to be part of the output of Rosenmann's algorithm.

\begin{algorithm}[H]
\caption{Extract-Coefficients\label{Algorithm: coefficient extraction}}
\begin{algorithmic}[1] 
\Statex{\textbf{Input:} A generating set $f_{1},f_{2},...,f_{k}\in \mathcal{A}$ for a right ideal $I$ and an element $f\in \mathcal{A}$.}
\Statex{\textbf{Output:} Coefficients $g_{1},g_{2},...,g_{m}\in  \mathcal{A}$ and $f'\in Span_{K}\left(T_{I}\right)$ such that $f=\sum_{i=1}^{m}\alpha_{i}g_{i}+f'$, where $\alpha_{1},\alpha_{2},...,\alpha_{m}$ is the Rosenmann basis of $I$ and $T_{I}$ its associated Schreier transversal.}
\vspace{2pt}
\State{Perform Rosenmann's algorithm on $f_{1},f_{2},...,f_{k}$}
\If{the output is the unity of $\mathcal{A}$}
	\State {\Return { $f$ as $g_{1}$ and $0$ as $f'$ } }
\EndIf
\State {save the output as $\alpha_{1},\alpha_{2},...,\alpha_{m},\beta_{1},\beta_{2},...,\beta_{m}$ and the matrix $C$ expressing $\overline{\beta}=\overline{\alpha} C$.}
\State {for every $1\leq i\leq m$ initialize $s_{i} \gets 0$ and $t_{i} \gets 0$}
\State {$f' \gets f$, $\text{changed\_}f' \gets True$}
\While{$\text{changed\_}f'$}
	\State {$\text{changed\_}f' \gets False$}
	\ForAll{ $\gamma \in \left\{ \alpha_{1},\alpha_{2},...,\alpha_{m},\beta_{1},\beta_{2},...,\beta_{m}\right\} $ }
		\State{reducible $ \gets \left\{ w\in\text{supp}\left(f'\right):\text{LPP}\left(\gamma\right)\text{ is a prefix of }w\right\} $}
		\If {reducible $\neq\emptyset$}
			\State{$w \gets \max\left\{ \text{reducible}\right\}$, $w' 
\gets \text{LPP}\left(\gamma\right)^{-1}w$}
			\State{$c \gets $coefficient of $w$ in $f'$}
			\State{$f' \gets f'-c\gamma w'$}
			\State{$\text{changed\_}f' \gets True$}
			\If {$\gamma==\alpha_{i}$ for some $1 \leq i \leq m$}
				\State {$s_{i} \gets s_{i} + cw'$}
			\EndIf
			\If {$\gamma==\beta_{i}$ for some $1 \leq i \leq m$}
				\State{$t_{i} \gets t_{i} + cw'$}
			\EndIf
		\EndIf
	\EndFor
\EndWhile
\State {$g\leftarrow s+Ct$, where $s=\left(s_{1},s_{2},...,s_{m}\right)^{T}$ and $t=\left(t_{1},t_{2},...,t_{m}\right)^{T}$}
\State{\Return}
\end{algorithmic}
\end{algorithm}

\begin{rem}
An explicit algorithm for computing the matrix $C$ used in Algorithm
\ref{Algorithm: coefficient extraction} is given in \cite[Alg.~5]{Seidel2025}.
In addition, \cite[Alg.~6]{Seidel2025} provides a method for extracting
the coefficients expressing an element with respect to the \emph{exposure
basis} of the ideal, extending Algorithm \ref{Algorithm: coefficient extraction}
beyond Rosenmann's shortlex framework to the wider class of exposure
orders.
\end{rem}

\subsection{From Ideals in $\mathcal{A}$ to Free $\mathcal{A}$-Modules\label{subsec: From Ideals to free modules}}

Suppose now that we are given a finite generating set $f_{1},f_{2},...,f_{t}$
for a submodule $M$ of the free right $\mathcal{A}$-module $\mathcal{A}^{k}$.
In this subsection we explain how to algorithmically (i) find a basis
for $M$ and (ii) express an element $f\in M$ in this basis. One
possible direction for achieving this is generalizing Rosenmann's
algorithm to submodules of free $\mathcal{A}$-modules. For brevity,
we take here a different course of action -- of embedding $M$ as
an ideal in $\mathcal{A}$ to perform all calculations inside the
group algebra. We note that this approach works only for free groups
of rank larger than $1$, as only in this case does the free module
$\mathcal{A}^{k}$ embed into $\mathcal{A}$, allowing us to realize
$M$ as an ideal.
\begin{rem}
\label{rem: In rkF=00003D1 coefficients can be extracted in a different way}In
the case that $F$ is the free group on a single generator, the ring
$\mathcal{A}=K\left[F\right]$ is the ring of Laurent polynomials
in a single variable over $K$, and in particular it is a (commutative)
Euclidean domain. Letting $Q$ be the $k\times t$ matrix whose columns
are $f_{1},f_{2},...,f_{t}$, we can algorithmically compute its column-style
Hermite Normal Form $H$, a form which generalizes the (reduced) column
echelon form (see \cite[Sec.~5.2]{Adkins2012}). The nonzero columns
of $H$ constitute a basis for $M$. The coefficients expressing any
$f\in M$ in this basis are then easily determined, one after the
other, with the $i$-th coefficient extracted by examining the row
containing the leading coefficient of the $i$-th basis element. This
method for extracting coefficients can then be used instead of Function
\ref{Algorithm: Express-In-Basis} in the algorithms we give for Theorem
\ref{thm: There exists algorithms for alg-free decomp}. However,
one can also compute the algebraic closure of the submodule $M$ in
$\mathcal{A}^{k}$ by using its Smith Normal Form (see \cite[Sec.~5.3]{Adkins2012}).
If $\text{rk}M=d$, this computable form gives an explicit automorphism
$\varphi$ of $\mathcal{A}^{k}$ such that $\varphi\left(M\right)$
is contained in the direct summand $\mathcal{A}^{d}\times\left\{ 0\right\} ^{k-d}$
of $\mathcal{A}^{k}$. Furthermore, $\varphi\left(M\right)$ cannot
be contained in a proper free factor of $\mathcal{A}^{d}\times\left\{ 0\right\} ^{k-d}$
by rank considerations over a PID. We obtain 
\[
\varphi\left(M\right)\leq_{\text{alg}}\mathcal{A}^{d}\times\left\{ 0\right\} ^{k-d}\leq_{*}\mathcal{A}^{k}\ \ \ \Longrightarrow\ \ \ M\leq_{\text{alg}}\varphi^{-1}\left(\mathcal{A}^{d}\times\left\{ 0\right\} ^{k-d}\right)\leq_{*}\mathcal{A}^{k}.
\]
For the algebraic closure of a general extension $M\leq N\leq\mathcal{A}^{k}$
as in Theorem \ref{thm: There exists algorithms for alg-free decomp},
expressing $M$ in some basis of $N$ reduces to the case $N=\mathcal{A}^{k}$.
\end{rem}
Suppose from here on that the rank of $F$ is at least $2$. To have
concrete ideals to work with, choose two distinct basis elements $b_{1}$
and $b_{2}$ of $F$ and let $H_{k}\leq F$ be the subgroup generated
(freely) by $\left\{ b_{2}^{-i}b_{1}b_{2}^{i}\right\} _{1\leq i\leq k}$.
Let $I_{k}\leq\mathcal{A}$ be the (right) ideal generated by $\left(b_{2}^{-i}b_{1}b_{2}^{i}-1\right)_{i\in\left\{ 1,2,..,k\right\} }$.
This generating set is in fact a basis for $I_{k}$ (see, for example,
\cite[Prop.~4.8 and Lem.~4.3]{Cohen2006}). Letting $\left(e_{i}\right)_{i=1}^{k}$
be the standard basis of $\mathcal{A}^{k}$ and $\left(e_{i}'\right)_{i=1}^{k}$
the Rosenmann basis for $I_{k}$, the $\mathcal{A}$-module homomorphism
$\varphi:\mathcal{A}^{k}\rightarrow I_{k}$ mapping $\varphi\left(e_{i}\right)=e_{i}'$
is an isomorphism. The following two functions implement $\varphi$
and its inverse:

\makeatletter
\renewcommand{\ALG@name}{Function}
\makeatother

\begin{algorithm}[H]
\noindent \caption{$\mathcal{\mathcal{A}}^{k}$-To-$I_{k}$\label{Function: ToGroupAlgebra}}

\begin{algorithmic}[1] 
\Statex{\textbf{Input:} An element $f=\sum_{i=1}^{k}e_{i}f_{i}\in \mathcal{A}^{k}$, where $\left(e_{i}\right)_{i=1}^{k}$ is the standard basis of $\mathcal{A}^{k}$}.
\Statex{\textbf{Output:} Its image $\varphi\left(f\right)$ under the isomorphism $\varphi: \mathcal{A}^{k}\rightarrow I_{k}$ mapping the standard basis  to the Rosenmann basis of $I_{k}$.}
\vspace{2pt}
\State{$e_{1}',e_{2}',...,e_{k}' \gets $ Rosenmann's Algorithm on $\left(b_{2}^{-i}b_{1}b_{2}^{i}-1\right)_{i\in\left\{ 1,2,..,k\right\}}$.}
\State{\Return{$\sum_{j=1}^{k}e_{j}'f_{j}$}.}
\end{algorithmic}
\end{algorithm}

\begin{algorithm}[H]
\noindent \caption{$I_{k}$-To-$\mathcal{A}^{k}$\label{Function: From-Group-Algebra}}

\begin{algorithmic}[1] 
\Statex{\textbf{Input:} An element $f'\in I_{k}$.}
\Statex{\textbf{Output:} An element $f\in \mathcal{A}^{k}$ such that $f'=\varphi\left(f\right)$ for the isomorphism $\varphi: \mathcal{A}^{k}\rightarrow I_{k}$ mapping the standard basis $\left(e_{i}\right)_{i=1}^{k}$ to the Rosenmann basis of $I_{k}$.}
\vspace{2pt}
\State{$g_{1},g_{2},...,g_{k},f'' \gets$ Extract-Coefficients$\left(\left(b_{2}^{-i}b_{1}b_{2}^{i}-1\right)_{i\in\left\{ 1,2,..,k\right\} }, f'\right)$. \# $f''$ should vanish.}
\State{\Return{$\sum_{j=1}^{k}e_{j}g_{j}$}.}
\end{algorithmic}
\end{algorithm}
For every submodule $M\leq\mathcal{A}^{k}$ we can pull back using
$\varphi^{-1}$ the Rosenmann basis of $\varphi\left(M\right)\leq\mathcal{A}$
to obtain a basis for $M$ which we also call its \emph{$\varphi$-Rosenmann
basis}. The following two functions obtain the $\mathbb{\varphi}$-Rosenmann
basis and express an element in this basis:

\begin{algorithm}[H]
\noindent \caption{Get-Basis\label{Algorithm: Get-Basis}}

\begin{algorithmic}[1] 
\Statex{\textbf{Input:} A generating set $f_{1},f_{2},...,f_{t}\in \mathcal{A}^{k}$ for a submodule $M$.}
\Statex{\textbf{Output:} The basis $\alpha_{1},\alpha_{2},...,\alpha_{\ell}$ for $M$ such that $\left(\varphi\left(\alpha_{i}\right)\right)_{i=1}^{\ell}$ is the Rosenmann basis for $\varphi\left(M\right)$.}
\State{$f_{1}',f_{2}',...,f_{t}' \gets$ respective outputs of $\mathcal{A}^{k}$-To-$I_k$ on $f_{1},f_{2},...,f_{t}$.}
\State{$\alpha_{1}',\alpha_{2}',...,\alpha_{\ell}' \gets$ Rosenmann basis from the output of Rosenmann's algorithm on $f_{1}',f_{2}',...,f_{t}'$.}
\State{\Return{respective outputs of $I_k$-To-$\mathcal{A}^{k}$ on $\alpha_{1}',\alpha_{2}',...,\alpha_{\ell}'$}.}
\end{algorithmic}
\end{algorithm}

\begin{algorithm}[H]
\noindent \caption{Express-In-Basis\label{Algorithm: Express-In-Basis}}

\begin{algorithmic}[1] 
\Statex{\textbf{Input}: A generating set $f_{1},f_{2},...,f_{t}\in \mathcal{A}^{k}$ for a submodule $M$ and an element $f\in M$.}
\Statex{\textbf{Output}: The coefficients $g_{1},g_{2},...,g_{\ell}$ such that $f=\sum_{i=1}^{\ell}\alpha_{i}g_{i}$ where $\alpha_{1},\alpha_{2},...,\alpha_{\ell}$ is the $\varphi$-Rosenmann basis of $M$.}
\vspace{2pt}
\State{$f',f_{1}',f_{2}',...,f_{t}' \gets$ respective outputs of $\mathcal{A}^{k}$-To-$I_k$ on $f,f_{1},f_{2},...,f_{t}$.}
\State{$g_{1},g_{2},...,g_{\ell} \gets$ Extract-Coefficients$\left(\left(f_{i}'\right)_{i\in\left\{ 1,2,...,t\right\} },f'\right)$}
\State{\Return{$g_{1},g_{2},...,g_{\ell}$}.}
\end{algorithmic}
\end{algorithm}

\makeatletter
\renewcommand{\ALG@name}{Algorithm}
\makeatother

\subsection{The Algorithms for the Algebraic Closure, Primitivity Testing and
Algebraicity Testing for Extensions of Free $\mathcal{A}$-Modules.}

The following algorithm computes, given a matrix $Q\in\text{Mat}{}_{k\times m}\left(\mathcal{A}\right)$
and a finitely generated column extension $R_{Q}\leq M$, the associated
$Q$-dual $\iota\left(M^{*Q}\right)$ as in Theorem \ref{Thm: Q-duality in the Free Group Algebra}:
\begin{algorithm}[H]
\noindent \caption{Get-$Q$-Dual\label{Algorithm: Get Q-Dual}}
\begin{algorithmic}[1] 
\Statex{\textbf{Input}: A matrix $Q\in \text{Mat}{}_{k\times m}\left(\mathcal{A}\right)$ and a generating set $f_{1},f_{2},...,f_{t}\in \mathcal{A}^{k}$ for a right submodule $M\leq \mathcal{A}^{k}$ of rank $t'$ containing the columns $q_{1},q_{2},...,q_{m}$ of $Q$.}
\Statex{\textbf{Output}: A generating set $\ell_{1},\ell_{2},...,\ell_{t'}$ for the dual $\iota\left(M^{*Q}\right)$, as in the duality of Theorem \ref{Thm: Q-duality in the Free Group Algebra}}
\vspace{2pt}
\ForAll{ $1 \leq i \leq m$}
	\State{$g_{1i},g_{2i},...,g_{t'i} \gets$ Express-In-Basis$\left(\left(f_{j}\right)_{j=1}^{t}, q_i \right)$.}
\EndFor
\State{$G \gets$ the matrix $\left(g_{ij}\right)_{i,j}$ of dimensions $t'\times m$.}
\State{\Return{the columns $\ell_{1},\ell_{2},...,\ell_{t'}$ of $\iota\left(G\right)$, where $\iota$ is the left-right inversion of Definition \ref{def: iota}.}}
\end{algorithmic}
\end{algorithm}

We are now ready to implement the algorithms of Theorem \ref{thm: There exists algorithms for alg-free decomp}.
We begin with the algebraic closure, computed by a double application
of the duality, according to Theorem \ref{Thm: Q-duality in the Free Group Algebra}:

\begin{algorithm}[H]
\noindent \caption{\label{alg: Get-Algebraic-Free}Get-Algebraic-Closure}

\begin{algorithmic}[1]
\Statex{\textbf{Input}: Respective generating sets $\left(f_{i}\right)_{i=1}^{m}$ and $\left(h_{i}\right)_{i=1}^{n}$ for right submodules $M$ and $N$ of $\mathcal{A}^{k}$ such that $M\leq N$.}
\Statex{\textbf{Output}: A generating set $\ell_{1},\ell_{2},...,\ell_{t}$ for the unique submodule $L$ satisfying $M\leq_{\text{alg}}L\leq_{*}N$.}
\vspace{4pt}
\State{Initialize $Q$ to be the matrix whose columns are $f_{1},f_{2},...,f_{m}$.}
\State{$g_{1},g_{2},...,g_{s} \gets $ Get-Q-Dual$\left(Q,\left(h_{i}\right)_{i=1}^{n}\right)$.}
\State{$\ell_{1},\ell_{2},...,\ell_{t} \gets $ Get-Q-Dual$\left(\iota\left(Q\right),\left(g_{i}\right)_{i=1}^{s}\right)$.}
\State{\Return{$\ell_{1},\ell_{2},...,\ell_{t}$}.}
\end{algorithmic}
\end{algorithm}

For the purpose of algebraicity testing, one may either perform a
$Q$-dual once and check that the rank has not decreased, as in Theorem
\ref{Thm: Equivalent conditions of algebraic extensions via duality in Introduction},
or, alternatively, use the algebraic closure as in Corollary \ref{Cor: freeness testing via duality}.
We implement here the second option.

\begin{algorithm}[H]
\noindent \caption{Is-Algebraic-Extension\label{alg: Is-Algebraic-Extension}}

\begin{algorithmic}[1] 
\Statex{\textbf{Input}: Respective generating sets $\left(f_{i}\right)_{i=1}^{m}, \left(h_{i}\right)_{i=1}^{n}$ for right submodules $M$ and $N$ of $\mathcal{A}^{k}$ such that $M\leq N$.}
\Statex{\textbf{Output}: True if $M\leq_{\text{alg}}N$, False otherwise.}
\vspace{4pt}
\State{$\ell_{1},\ell_{2},...,\ell_{t} \gets $ Get-Algebraic-Closure($\left(f_{i}\right)_{i=1}^{m},\left(h_{i}\right)_{i=1}^{n}$).}
\State{\Return{Get-Basis($\left(\ell_{i}\right)_{i=1}^{t}$) $==$ Get-Basis($\left(h_{i}\right)_{i=1}^{n}$)}.}
\end{algorithmic}
\end{algorithm}
Note that we have used here the uniqueness of the Rosenmann basis
to test if two generating sets generate the same submodule. Finally,
for testing if an extension is free, we use again the algebraic closure,
as in Corollary \ref{Cor: freeness testing via duality}:

\begin{algorithm}[H]
\noindent \caption{Is-Free-Extension\label{alg: Is-Free-Extension}}

\begin{algorithmic}[1] 
\Statex{\textbf{Input}: Respective generating sets $\left(f_{i}\right)_{i=1}^{m}, \left(h_{i}\right)_{i=1}^{n}$ for right submodules $M$ and $N$ of $\mathcal{A}^{k}$ such that $M\leq N$.}
\Statex{\textbf{Output}: True if $M\leq_{*}N$, False otherwise.}
\vspace{4pt}
\State{$\ell_{1},\ell_{2},...,\ell_{t} \gets $ Get-Algebraic-Closure($\left(f_{i}\right)_{i=1}^{m},\left(h_{i}\right)_{i=1}^{n}$).}
\State{\Return{Get-Basis($\left(\ell_{i}\right)_{i=1}^{t}$) $==$ Get-Basis($\left(f_{i}\right)_{i=1}^{m}$)}.}
\end{algorithmic}
\end{algorithm}

\begin{rem}
A previous result of the current authors \cite[Cor.~5.3]{ErnstWest2024}
describes an algorithm testing for primitivity in the free group algebra
in certain scenarios. Given some $w\in F$ and an ideal $I\leq\mathcal{A}$
of rank $2$ containing $w-1$, all possible elements supported on
the subtree $\left[1,w\right]$ of the Cayley graph of $F$ (with
respect to some free generating set) are iterated in search of an
element extending $w-1$ to a basis of $I$. The algorithms in the
current paper present significant improvements. First, the algebraic
closure is computed, which allows for testing not only if $M$ is
itself a free factor of $N$ but also if $M$ is contained in \emph{some}
proper free factor of $N$. Second, the new algorithms have a wider
applicability: (i) they allow the larger ideal $N$ to be a submodule
of $\mathcal{A}^{k}$ and not only an ideal in $\mathcal{A}$, and
of arbitrary finite rank instead of rank $2$; (ii) they allow the
smaller submodule $M$ to be of arbitrary finite rank instead of a
principal ideal generated by an element of the form $w-1$; and (iii)
they allow the base field $K$ to be infinite, so long as its operations
are computable (for example, $K=\mathbb{Q}$). Finally, in the previous
algorithm the worst-case amount of calls to Rosenmann's algorithm
is exponential in the length of $w$ -- once for every element supported
on $\left[1,w\right]$. The current algorithms only require a constant
number of calls to Rosenmann algorithm. For example, the algebraic
closure of an extension $M\leq N$ requires two calls to Rosenmann's
algorithm: once for each $Q$-dual calculated (the Rosenmann basis
of the embedded ideal $I_{k}$ can be precalculated).
\end{rem}

\section{\label{sec: algs for free groups}An Algorithm for the Algebraic
Closure in Free Groups}

In this section we explain how the algorithms for free group algebras,
described in the previous section, can be applied to obtain analogous
algorithms for extensions of free groups. 

Similarly to Definition \ref{def: primitivity, free factor, algebraic, algebraic-free decomposition},
an extension $H\leq H'$ of free groups is called\emph{ free} if a
basis for $H$ can be extended to a basis of $H'$, in which case
$H$ is called a \emph{free factor} of $H'$ and we write $H\leq_{*}H'$.
The extension is \emph{algebraic} if $H$ is not contained in any
proper free factor of $H'$, in which case we write $H\leq_{\text{alg}}H'.$
If $\left\langle w\right\rangle \leq_{*}H'$ for some nontrivial $w\in H'$
we say that $w$ is \emph{primitive} in $H'$. If at least one of
$H$ or $H'$ is finitely generated then, analogously to Theorem \ref{Theorem: existence and uniqueness of Algebraic - free decomposition},
there exists a unique algebraic extension $L$ of $H$ which is a
free factor of $H'$, namely, a unique $L$ with $H\leq_{\text{alg}}L\leq_{*}H'$.
This $L$ is called the \emph{algebraic closure} of the extension
$H\leq H'$. This intermediate subgroup further satisfies $L=\bigcup_{H\leq_{\text{alg}}\widetilde{L}\leq H'}\widetilde{L}=\bigcap_{H\leq\widetilde{L}\leq_{*}H'}\widetilde{L}$
(see \cite[Thm.~3.16]{Miasnikov2007}).

Various algorithms exist for testing the freeness or algebraicity
of an extension $H\leq H'$ of finitely generated free groups, as
well as for finding its algebraic closure. The best known and often
most efficient approach is based on a classical work of Whitehead
\cite{Whitehead1936}. It involves repetitive applications of a specific
finite set of automorphisms -- the Whitehead automorphisms of $H'$.
Each application simplifies the complexity of $H$ with respect to
the current basis of $H'$ (see \cite{Gersten1984}), until $H$ is
either declared algebraic in $H'$ or is contained in a proper free
factor of $H'$. Repeating the process and keeping track of the automorphisms
used yields the algebraic closure of $H\leq H'$. 

A complementary approach \cite[Lem.~11.8]{Kapovich2002} combines
Whitehead's algorithm -- used only for detecting free factors --
with the fact that the set of algebraic extensions of $H$ in $H'$
is finite and can be computed via identifications of vertices in the
\emph{core graph} of $H$. A more detailed analysis of such graph
identifications \cite[Appendix~A]{Puder2014} yields an alternative
to Whitehead's algorithm for detecting free factors and algebraicity.

The approach presented here is based on the idea of considering specific
ideals, $J_{H}$ and $J_{H'}$, representing $H$ and $H'$, respectively\footnote{This technique has already given rise to (non-algorithmic) criteria
for primitivity and algebraicity (see Umirbaev's \cite{Umirbaev1994}
and \cite{Umirbaev1996} and Cohen's \cite[Thm.~D]{Cohen2006}). Our
approach will provide explicit algorithms, including a computation
of the algebraic closure. }. The ideals are in ${\cal A}=K[F]$, where $K$ is a field that can
be chosen arbitrarily. Following Cohen's \cite[Sec.~4]{Cohen2006},
for every subgroup $H\leq F$ let $J_{H}\leq\mathcal{A}$ be the right
ideal generated by the elements $\left\{ h-1\right\} _{h\in H}$.
When $H=F$, the ideal $J_{F}$ is called the \emph{augmentation ideal}
of $\mathcal{A}$ as it is the kernel of the \emph{augmentation map}
$\varepsilon:\mathcal{A}\rightarrow K$ defined by $\varepsilon\left(\sum_{w\in F}\alpha_{w}w\right)=\sum_{w\in F}\alpha_{w}$.
We recall the following facts from \cite{Cohen2006} (Lemmas 4.1 and
4.2, the example following Lemma 4.3, and Theorems 4.7 and D):
\begin{thm}[\textbf{\emph{Cohen}}]
\textbf{\emph{\label{Thm Cohen}}} Let $F$ be a free group, $K$
a field and $H,H'\leq F$ subgroups. Then:
\begin{enumerate}
\item For every $h\in F$ we have $h\in H\iff$$h-1\in J_{H}$.
\item Elements $\left(h_{\alpha}\right)_{\alpha\in I}$ of $F$ form a generating
set for $H$ if and only if the elements $\left(h_{\alpha}-1\right)_{\alpha\in I}$
form a generating set for $J_{H}.$ In particular, $rk\left(H\right)=rk\left(J_{H}\right).$
\item $H\leq_{*}H'$ if and only if $J_{H}\leq_{*}J_{H'}$.
\end{enumerate}
\end{thm}
In particular, the freeness of an extension of free groups implies
the freeness of the extension of their associated ideals. It is also
clear from Theorem \ref{Thm Cohen} that if $J_{H}\leq_{\text{alg}}J_{H'}$
then $H\leq_{\text{alg}}H'$. Proposition \ref{prop: algebraicity transfers to group algebra}
below shows the converse also holds and thus, as stated in Theorem
\ref{thm: algebraic-free decomposition transfers to group algebra},
if $L$ is the algebraic closure of $H\le H'$, then $J_{L}$ is the
algebraic closure of $J_{H}\le J_{H'}$. We need the following result
of Umirbaev \cite{Umirbaev1996}, which we restate in the language
of $Q$-duality:
\begin{thm}
\label{Thm Umirbaev}\textbf{(Umirbaev, restated)} Let $H$ be a finitely
generated subgroup of a free group $F$, and let $Q=\left(h_{1}-1,h_{2}-1,...,h_{t}-1\right)\in\text{Mat}{}_{1\times t}\left(\mathcal{A}\right)$
for some generating set $h_{1},h_{2},...,h_{t}$ of $H$. Then the
$Q$-dual of the augmentation ideal $J_{F}\leq\mathcal{A}$ satisfies
$\text{rk}\left(J_{F}^{*Q}\right)=\text{rk}\left(L\right)$ where
$H\leq_{\text{alg}}L\leq_{*}F.$
\end{thm}
Note that Umirbaev's result gives what we need when $H'=F$ is the
ambient group: in this case, Theorems \ref{Thm Cohen}, \ref{Thm Umirbaev}
and the results of Section \ref{sec:The-Duality-Induced} yield that
$J_{L}$ is an intermediate free factor $J_{H}\le J_{L}\le_{*}J_{H'}$
of the same rank as the algebraic closure, and therefore must coincide
with the algebraic closure. However, we want to establish this result
for an arbitrary $H'$. 
\begin{lem}
\label{lem: K=00005BH=00005D-ind implies K=00005BF=00005D ind for elements of K=00005BH=00005D^m}
Let $H\leq F$, $m\in\mathbb{Z}_{\geq1}$ and $B$ a subset of elements
of the free left $K\left[H\right]$-module $K\left[H\right]^{m}.$
Let $M_{H}$ and $M_{F}$ be the respective submodules of $K\left[H\right]^{m}$
and $K\left[F\right]^{m}$ generated by $B$. Then:
\begin{enumerate}
\item $B$ is a $K\left[H\right]$-linearly independent set if and only
if it is $K\left[F\right]$-linearly independent.
\item $M_{F}\cap K\left[H\right]^{m}=M_{H}$.
\item $\text{rk}M_{H}=\text{rk}M_{F}$.
\end{enumerate}
\end{lem}
\begin{proof}
Clearly, $K\left[F\right]$-linear independence implies $K\left[H\right]$-linear
independence. For the other implication, for every left coset $C$
of $H$ in $F$, choose a representative $w_{C}\in C$ (so that $C=w_{C}H$),
and let $P_{C}:K\left[F\right]\rightarrow K\left[H\right]$ be defined
by
\[
P_{C}\left(\sum_{u\in F}\alpha_{u}u\right)=w_{C}^{-1}\left(\sum_{u\in C}\alpha_{u}u\right).
\]
For every $h\in H$ and $u\in F$ we have $u\in C\iff uh\in C$ and
so $P_{C}$ is a homomorphism of \uline{right} $K\left[H\right]$-modules.
Applying this property coordinate-wise, the map $\left(P_{C}\right)^{m}:K\left[F\right]^{m}\rightarrow K\left[H\right]^{m}$
satisfies for every $g\in K\left[F\right]$ and $f\in K\left[H\right]^{m}$
that $\left(P_{C}\right)^{m}\left(gf\right)=P_{C}\left(g\right)f.$
Now let $\sum_{f\in B}g_{f}f=0$ for some coefficients $g_{f}\in K\left[F\right]$
where all but finitely many of which vanish. Applying $\left(P_{C}\right)^{m}$
to both sides gives $\sum_{f\in B}P_{C}\left(g_{f}\right)f=0$, which
by $K\left[H\right]$-linear independence implies that $P_{C}\left(g_{f}\right)=0$
for all $f\in B$. Thus, $g_{f}=\sum_{C\in F/H}w_{C}P_{C}\left(g_{f}\right)=0$.
\\
For the second assertion, clearly $M_{H}\subseteq M_{F}\cap K\left[H\right]^{m}$.
Suppose that $a\in M_{F}\cap K\left[H\right]^{m}$. Since $a\in M_{F}$,
it can be expressed as $a=\sum_{f\in B}g_{f}f$ for some coefficients
$g_{f}\in K\left[F\right]$ all but finitely many of which vanish.
Choosing $w_{H}$ to be the trivial word and applying $\left(P_{H}\right)^{m}$
to both sides gives $a=\sum_{f\in B}P_{H}\left(g_{f}\right)f\in M_{H}$.\\
For the third assertion, let $B'$ be a basis for $M_{H}$. Then $B'$
is $K\left[H\right]$-linearly independent. By the first assertion,
it is a basis for the submodule it generates over $K\left[F\right]$.
But $B$ and $B'$ generate the same submodule over $K\left[H\right]$
and, hence, also over $K\left[F\right]$. Thus $B'$ is a basis for
$M_{F}$ too, and the claim follows.
\end{proof}
\begin{prop}
\label{prop: algebraicity transfers to group algebra}Let $H$ and
$H'$ be subgroups of a free group $F$. If $H\leq_{\text{alg}}H'$
then $J_{H}\leq_{\text{alg}}J_{H'}$.
\end{prop}
\begin{proof}
Let $Q=\left(h_{1}-1,h_{2}-1,...,h_{t}-1\right)$ and $Q'=\left(h_{1}'-1,h_{2}'-1,...,h_{m}'-1\right)$,
where $h_{1},...,h_{t}$ and $h_{1}',...,h_{m}'$ are respective bases
for $H$ and $H'$. We work first over $K\left[H'\right]$. By Theorem
\ref{Thm Umirbaev}, for the augmentation ideal $I_{H'}$ of $K\left[H'\right]$
we have $\text{rk}\left(I_{H'}^{*Q}\right)=\text{rk}\left(H'\right)=m$.
Express $Q=Q'B$ for some $B\in\text{Mat}{}_{m\times t}\left(K\left[H'\right]\right)$.
By Lemma \ref{lem: K=00005BH=00005D-ind implies K=00005BF=00005D ind for elements of K=00005BH=00005D^m},
the rows of $B$ generate ideals of the same rank over $K\left[H'\right]$
and over $K\left[F\right]$, so $\text{rk}\left(J_{H'}^{*Q}\right)=m$
as well. Thus, taking the $Q$-dual for $J_{H'}$ does not decrease
rank, which by Theorem \ref{Thm: Equivalent conditions of algebraic extensions via duality in Introduction}
implies $J_{H}\leq_{\text{alg}}J_{H'}$. 
\end{proof}
By combining Theorem \ref{Thm Cohen} with Proposition \ref{prop: algebraicity transfers to group algebra},
we obtain the following:
\begin{thm}
\label{thm: algebraic-free decomposition transfers to group algebra}Let
$H\leq H'$ be an extension of finitely generated subgroups of the
free group $F$ and let $K$ be a field. If $L$ is the algebraic
closure of $H\leq H'$ then $J_{L}$ is the algebraic closure of the
associated extension $J_{H}\leq J_{H'}$ of right ideals in $K\left[F\right]$.
\end{thm}
\begin{rem}
Let $F$ be a free group and $H\le F$ a finitely generated subgroup.
Consider the set $AE(J_{H})$ of proper algebraic extensions of $J_{H}$
in $K[F]$, excluding $K[F]$ itself. Proposition \ref{prop: algebraicity transfers to group algebra}
shows that the set of ranks of the elements in $AE(J_{H})$ contains
the set of ranks of the elements in $AE(H)$, the set of proper algebraic
extensions of $H$ in $F$. It is unclear whether the minimum rank
in both sets is identical -- see, e.g., \cite[Conj.~1.9]{ErnstWest2024}.
\end{rem}
It is now clear that given subgroups $H\le H'\le F$, we can obtain
(a generating set for) $J_{L}$ where $L$ is the algebraic closure
of $H\le H'$. It still remains to explain how one can algorithmically
recover $L$ from $J_{L}$. We generalize to arbitrary fields a method
proposed by Rosenmann \cite[Sec.~7]{Rosenmann1993} for the field
of order 2. 
\begin{prop}
\label{prop:(Rosenmann) rosenmann basis of group ideals is binomials}Let
$H\leq F$ be finitely generated. Then the Rosenmann basis $\alpha_{1},\alpha_{2},...,\alpha_{m}$
for the right ideal $J_{H}\leq\mathcal{A}$ generated by $\left\{ h-1\right\} _{h\in H}$
comprises of binomials of the form $\alpha_{i}=u_{i}-v_{i}$ where
$u_{i},v_{i}\in F$. Moreover, the elements $\left(u_{i}v_{i}^{-1}\right)_{1\leq i\leq m}$
form a basis for $H$.
\end{prop}
\begin{proof}
The ideal $J_{H}$ has a finite basis $B$ consisting only of binomials
(elements of $\mathcal{A}$ with support of cardinality 2, see Theorem
\ref{Thm Cohen}). We apply Rosenmann's algorithm to this particular
basis $B$ of $J_{H}$. Each step of the algorithm either (i) multiplies
some $f\in B$ by a unit ($f\leftarrow f\cdot\alpha w$ where $\alpha\in K$
and $w\in F$) or (ii) adds to some $f\in B$ another $f'\in B$ ($f\leftarrow f+f'$)
such that cancellation occurs, i.e., such that there exists some $w\in\text{supp}\left(f\right)\backslash\text{supp}\left(f+f'\right)$.\footnote{In Rosenmann's language, ``Orbit Reductions'' only require multiplication
by a unit and the other reduction types require some combination of
steps of type (i) and (ii).} Multiplication by a unit clearly preserves the cardinality of the
support. For the second type, if $f$ and $f'$ are both binomials
and cancellation occurs in $f+f'$ then the cardinality of the support
of $f+f'$ is at most $2$, and in fact, must be exactly $2$: It
cannot be $1$ since $J_{H}$ is a proper ideal as it is contained
in the augmentation ideal $J_{F}$. It cannot be $0$ since this would
imply that the Rosenmann basis for $J_{H}$ has fewer elements than
the basis $B$. Thus, the Rosenmann basis of $J_{H}$ must be of the
form $\alpha_{1},\alpha_{2},...,\alpha_{m}$ for $\alpha_{i}=u_{i}+\lambda_{i}v_{i}$
where $u_{i},v_{i}\in F$ and $\lambda_{i}\in K^{*}$. As $\alpha_{i}\in J_{H}\subseteq J_{F}$,
we must have $\lambda_{i}=-1$. Multiplying each $f_{i}$ by the unit
$v_{i}^{-1}$, we obtain another basis $\left(u_{i}v_{i}^{-1}-1\right)_{1\leq i\leq m}$
for $J_{H}$. By Theorem \ref{Thm Cohen}, the elements $\left(u_{i}v_{i}^{-1}\right)_{1\leq i\leq m}$
form a basis for $H$.
\end{proof}

Using Theorem \ref{thm: algebraic-free decomposition transfers to group algebra}
and Proposition \ref{prop:(Rosenmann) rosenmann basis of group ideals is binomials},
we are now ready to compute the algebraic closure of an extension
of finitely generated free groups. For efficiency we use the field
of order $2$ but the algorithm works over other fields as well:

\begin{algorithm}[H]
\noindent \caption{\label{alg:Get-Algebraic-Free-For-Groups}Get-Algebraic-Closure-For-Free-Groups}

\begin{algorithmic}[1]
\Statex{\textbf{Input}: Respective generating sets $\left(h_{i}\right)_{i=1}^{m}$ and $\left(h_{i}'\right)_{i=1}^{n}$ for subgroups $H$ and $H'$ of a free group $F$ such that $H\leq H' $.}
\Statex{\textbf{Output}: A basis for the unique intermediate subgroup $L$ such that $H\leq_{\text{alg}}L\leq_{*}H'$.}
\vspace{4pt}
\State{$\ell_{1},\ell_{2},...,\ell_{t} \gets $ Get-Algebraic-Closure$\left(\left(h_{i}-1\right)_{i=1}^{m},\left(h_{i}'-1\right)_{i=1}^{n}\right)$ over the field of order $2$.}
\State{Apply Rosenmann's algorithm to $\ell_{1},\ell_{2},...,\ell_{t}$ to get the Rosenmann basis $u_{1}-v_{1},u_{2}-v_{2},...,u_{s}-v_{s}$.}
\State{\Return{$u_{1}v_{1}^{-1},u_{2}v_{2}^{-1},...,u_{s}v_{s}^{-1}$}.}
\end{algorithmic}
\end{algorithm}

Using any method for checking if two finite sets of elements of $F$
generate the same subgroup (for example, using Stallings foldings
and comparing the core graphs obtained, see \cite{Kapovich2002} for
details), we obtain algorithms for testing whether a given extension
of finitely generated free groups is algebraic and whether it is free:

\begin{algorithm}[H]
\noindent \caption{\label{alg:Is-Algebraic-Extension-For-Group}Is-Algebraic-Extension-For-Free-Groups}

\begin{algorithmic}[1] 
\Statex{\textbf{Input}: Respective generating sets $\left(h_{i}\right)_{i=1}^{m}$ and $ \left(h_{i}'\right)_{i=1}^{n}$ for subgroups $H$ and $H'$ of a free group $F$ such that $H\leq H' $.}
\Statex{\textbf{Output}: True if $H\leq_{\text{alg}}H'$, False otherwise.}
\vspace{4pt}
\State{$u_{1},u_{2},...,u_{s} \gets $ Get-Algebraic-Closure-For-Free-Groups($\left(h_{i}\right)_{i=1}^{m},\left(h_{i}'\right)_{i=1}^{n}$).}
\State{\Return{$\left\langle u_{1},u_{2},...,u_{s}\right\rangle ==\left\langle h_{1}',h_{2}',...,h_{n}'\right\rangle $}.}
\end{algorithmic}
\end{algorithm}

\begin{algorithm}[H]
\noindent \caption{\label{alg:Is-Free-Extension-For-Groups}Is-Free-Extension-For-Free-Groups}

\begin{algorithmic}[1] 
\Statex{\textbf{Input}: Respective generating sets $\left(h_{i}\right)_{i=1}^{m}$ and $ \left(h_{i}'\right)_{i=1}^{n}$ for subgroups $H$ and $H'$ of a free group $F$ such that $H\leq H' $.}
\Statex{\textbf{Output}: True if $H\leq_{*}H'$, False otherwise.}
\vspace{4pt}
\State{$u_{1},u_{2},...,u_{s} \gets $ Get-Algebraic-Closure-For-Free-Groups($\left(h_{i}\right)_{i=1}^{m},\left(h_{i}'\right)_{i=1}^{n}$).}
\State{\Return{$\left\langle u_{1},u_{2},...,u_{t}\right\rangle ==\left\langle h_{1},h_{2},...,h_{r}\right\rangle $}.}
\end{algorithmic}
\end{algorithm}

\section{\label{sec: intersection algorithm}An Algorithm for the Intersection
of Free Modules over a Free Group Algebra}

Let $K$ be a field and $F$ a free group of rank at least $2$, and
denote by $\mathcal{A}=K\left[F\right]$ the associated free group
algebra. Let $M$ and $N$ be finitely generated submodules of the
free right $\mathcal{A}$-module $\mathcal{A}^{k}$. We give here
an algorithm for computing a free basis for the intersection $M\cap N$.
As discussed in Section \ref{subsec: From Ideals to free modules},
the rank condition ensures that $\mathcal{A}^{k}$ embeds into $\mathcal{A}$,
allowing submodules to be realized as ideals. In rank $1$, classical
methods based on~Hermite normal form can be applied instead (see
Remark \ref{rem: In rkF=00003D1 coefficients can be extracted in a different way}).

The algorithm we present is a variant of the well-known Zassenhaus'
algorithm from linear algebra which, given subspaces $U$ and $V$
of a finite-dimensional vector space, calculates linear bases for
the sum $U+V$ and for the intersection $U\cap V$. Our variant bootstraps
the ability to calculate a basis for the sum (via Rosenmann's algorithm)
to obtain a basis for the intersection. The main idea in play here
is that the intersection $M\cap N$ and the sum $M+N$ appear together
in the short exact sequence of right $\mathcal{A}$-modules 
\[
0\rightarrow M\cap N\overset{j}{\rightarrow}M\oplus N\overset{\pi}{\rightarrow}M+N\rightarrow0,
\]
where $M\oplus N=\left\{ \left(x,y\right):x\in M,y\in N\right\} $
and the maps are $j\left(x\right)=\left(x,-x\right)$ and $\pi\left(x,y\right)=x+y$.
The module $M+N$ is a free $\mathcal{A}$-module (as a submodule
of $\mathcal{A}^{k}$), so the short exact sequence splits. In particular,
$j\left(M\cap N\right)$ is a direct summand of $M\oplus N$ so one
can project a generating set of $M\oplus N$ onto $j\left(M\cap N\right)$
and apply $j^{-1}$ to obtain a generating set for $M\cap N$.

We implement this idea in the following way. We first use some fixed
embedding $\varphi:\mathcal{A}^{k}\hookrightarrow\mathcal{A}$ (see
Function \ref{Function: ToGroupAlgebra}) to perform all calculations
inside the group algebra. Let $I$ and $J$ denote the respective
images of $M$ and $N$ under $\varphi$. Let $f_{1},...,f_{m}$ and
$g_{1},...,g_{n}$ be respective generating sets for $I$ and $J$,
and let $h_{1},h_{2},...,h_{r}$ be the Rosenmann basis of $I+J$.
Let $v\in\mathcal{A}^{m+n}$ and $u\in\mathcal{A}^{r}$ be the row
vectors whose entries are the respective generating sets $f_{1},f_{2},...,f_{m},g_{1},g_{2},...,g_{n}$
and $\gamma_{1},\gamma_{2},...,\gamma_{r}$ for $I+J$. Since both
sets generate the same ideal, there exist matrices $A\in\text{Mat}_{\left(m+n\right)\times r}\left(\mathcal{A}\right)$
and $B\in\text{Mat}_{r\times\left(m+n\right)}\left(\mathcal{A}\right)$
such that $vA=u$ and $uB=v$. Let $P$ and $Q$ be the block matrices
in $\text{Mat}_{\left(m+n\right)\times\left(m+n\right)}\left(\mathcal{A}\right)$
defined by $P=\begin{pmatrix}I_{m} & 0\\
0 & 0
\end{pmatrix}$ and $Q=\begin{pmatrix}0 & 0\\
0 & I_{n}
\end{pmatrix}$, where $I_{q}$ denotes the $q\times q$ identity matrix. The following
claim shows that these matrices allow us to obtain a generating set
for $I\cap J$. 
\begin{prop}
The entries of the vector $vP\left(I_{m+n}-AB\right)\in\mathcal{A}^{m+n}$
form a generating set for $I\cap J$.
\end{prop}
\begin{proof}
Let $L\leq\mathcal{A}$ be the ideal generated by the entries of $vP\left(I_{m+n}-AB\right)$.
Then $L$ is the column space of $vP\left(I_{m+n}-AB\right)$, i.e.,
the set of elements of the form $vP\left(I_{m+n}-AB\right)t$, where
$t\in\mathcal{A}^{m+n}$ is some column vector. Since $vP=\left(f_{1},f_{2},...,f_{m},0,...0\right)$,
by Claim \ref{matrix product Column space contained}, the ideal $L$
is contained in the column space of $vP$, which is $I$. As $P+Q=I_{m+n}$,
\[
v\left(P+Q\right)\left(I_{m+n}-AB\right)=v\left(I_{m+n}-AB\right)=v-uB=v-v=0.
\]
Thus, $vP\left(I_{m+n}-AB\right)=vQ\left(AB-I_{m+n}\right)$, and
another use of Claim \ref{matrix product Column space contained}
now gives that $L$ is contained in the column space of $vQ=\left(0,0,...,0,g_{1},g_{2},...,g_{n}\right)$,
which is $J$. Hence, $L\subseteq I\cap J$. Conversely, suppose that
$x\in I\cap J$. Then we may express $x=\sum_{i=1}^{m}f_{i}p_{i}$
and $x=\sum_{i=1}^{n}g_{i}q_{i}$ for some $p_{1},p_{2},...,p_{m},q_{1},q_{2},...,q_{n}\in\mathcal{A}$.
Then for the column vector $t=\left(p_{1},p_{2},...,p_{m},-q_{1},-q_{2},...,-q_{n}\right)^{T}\in\mathcal{A}^{m+n}$
we have $vPt=x$ and $vt=0$, and so $vP\left(I_{m+n}-AB\right)t=vPt-vABt=x-vt=x.$
Hence, $I\cap J\subseteq L$.
\end{proof}
The matrices $A$ and $B$ can be computed as follows. The matrix
$B$ is uniquely determined by the condition $uB=v$. The $i$-th
column in $B$ is composed of the $\mathcal{A}$-coefficients expressing
the $i$-th entry of $v$ in the Rosenmann basis $h_{1},h_{2},...,h_{r}$.
Hence, $B$ can be computed column by column using Algorithm \ref{Algorithm: coefficient extraction}.
As for $A$, it is generally not uniquely determined by the condition
$vA=u$ since the entries of $v$ are not necessarily a basis for
the ideal $I+J$. We can nonetheless compute $A$ by keeping track
of the steps performed when obtaining $u$ from $v$ during Rosenmann's
algorithm. Recall that the algorithm begins with $v_{0}:=v\in\mathcal{A}^{m+n}$.
The vector $v_{i}$ at the $i$-th step is obtained from $v_{i-1}$
by either multiplying one of its entries by a unit of $\mathcal{A}$
or by adding one of its entries to another. Equivalently, $v_{i-1}$
is multiplied from the right by the appropriate elementary matrix
$E_{i}\in\text{GL}_{m+n}\left(\mathcal{A}\right)$ to obtain $v_{i}=v_{i-1}E_{i}$.
If there are $\ell$ steps in the algorithm, the Rosenmann basis $u$
is obtained from $v_{\ell}$ by removing entries equal to $0$ and
possibly reordering. Equivalently, $u=v_{\ell}E_{\text{proj}}$, where
$E_{\text{proj}}\in\text{Mat}_{\left(m+n\right)\times r}\left(\mathcal{A}\right)$
is the corresponding projection matrix. Thus, $u=vE_{1}E_{2}...E_{\ell}E_{\text{proj}}$,
so we may choose $A=E_{1}E_{2}...E_{\ell}E_{\text{proj }}.$ If, during
Rosenmann's algorithm, we perform in parallel the same sequence of
steps on the columns of the identity matrix $I_{m+n}$ (instead of
$v_{0}$), we will end up with the desired matrix $I_{m+n}E_{1}E_{2}...E_{\ell}E_{\text{proj}}=A$.
We thus obtain the following algorithm:

\begin{algorithm}[H]
\noindent \caption{\label{alg:Intersection_of_modules}Find-Free-Module-Intersection}

\begin{algorithmic}[1] 
\Statex{\textbf{Input}: Respective generating sets $\left(f_{i}'\right)_{i=1}^{m}$ and $ \left(g_{i}'\right)_{i=1}^{n}$ for respective submodules $M$ and $N$ of the free right $\mathcal{A}$-module $\mathcal{A}^{k}$.}
\Statex{\textbf{Output}: A generating set $\left(h_{i}'\right)_{i=1}^{m+n}$ for the submodule $M\cap N$.}
\vspace{4pt}
\State{$f_{1},f_{2},...,f_{m},g_{1},g_{2},...,g_{n} \gets$ respective outputs of $\mathcal{A}^{k}$-To-$I_k$ on $f_{1}',f_{2}',...,f_{m}',g_{1}',g_{2}',...,g_{n}'$.}
\State{Initialize $v$ and $v_{cur}$ to be two copies of the row vector $\left(f_{1},f_{2},...,f_{m},g_{1},g_{2},...,g_{n}\right)\in \mathcal{A}^{m+n}$.}
\State{Initialize $A_{cur}$ to be the identity matrix of dimension $m+n$ over $\mathcal{A}$.}
\State{Perform Rosenmann's algorithm on $v_{cur}$ while performing each step on $A_{cur}$ as well: In each step perform $v_{cur}\leftarrow v_{cur}\cdot E$, where $E$ is the appropriate matrix, followed by $A_{cur}\leftarrow A_{cur}\cdot E$.}
\State{Initialize $A \gets A_{cur}$, $u \gets v_{cur}$ and $r \gets length(u)$.}
\ForAll{ $1 \leq i \leq m+n$}
	\State{$B_{1i},B_{2i},...,B_{ri},v_{i}' \gets$ Extract-Coefficients$\left(\left(u_{j}\right)_{j\in\left\{ 1,2,..,r\right\} }, v_{i}\right)$. \# $v_{i}'$ should vanish.}
\EndFor
\State{Initialize $B$ to be the matrix $\left(B_{ij}\right)_{i,j}$ of dimensions $r\times\left(m+n\right)$.}
\State{Initialize $P$ to be the square matrix of dimension $m+n$ given in block form by $\begin{pmatrix}I_{m} & 0\\0 & 0\end{pmatrix}$.}
\State{$h \gets vP\left(I_{m+n}-AB\right)$.}
\State{\Return{$\left(\ensuremath{I_{k}}\text{-To-}\ensuremath{\mathcal{A}^{k}}\left(h_{i}\right)\right)_{i=1}^{m+n}$}.}
\end{algorithmic}
\end{algorithm}

\section{Further Algorithms?\label{sec:Further-Algorithms?}}

We list several algorithmic questions relating to the current paper
which, to the best of our knowledge, are still open.
\begin{enumerate}
\item \textbf{Automorphic equivalence} \textbf{in free group algebras:}
Let $K$ be a field and $F$ a free group. Is there an algorithm for
deciding, given elements $f$ and $g$ in $K\left[F\right]$, whether
there exists a $K$-algebra automorphism $\varphi$ of $K\left[F\right]$
such that $\varphi\left(f\right)=g$? This would be analogous to Whitehead's
second algorithm, which settles this question inside a free group
\cite{Whitehead1936a}. 
\item \textbf{Algorithms for group algebras of surface groups:} Surface
groups (of negative Euler characteristic) share many commonalities
with free groups. For such a group $G$ and a field $K$, can we find
analogous algorithms for:
\begin{enumerate}
\item \textbf{Membership in an ideal:} Can we determine, given a finite
subset $B$ of $K\left[G\right]$ and an element $f\in K\left[G\right]$,
if $f$ lies in the right ideal generated by $B$? This is achieved
by Rosenmann's algorithm in free group algebras. We remark that there
do exist algorithms for the analogous problem for subgroups of $G$,
at least when $G$ is an orientable surface group (see \cite{Magee2022}
for one such algorithm).
\item \textbf{Detecting free ideals:} It is shown in \cite[Cor.~3]{Avramidi2022}
that if $G$ is an orientable surface group of high enough genus,
then a submodule of $(K[G])^{d}$ which is generated by two elements
is necessarily free. (See also \cite{Avramidi2023} for recent extensions
of this result.) Given a finite subset $B$ of $K[G]$, is there an
algorithm to determine if the ideal generated by $B$ is a free right
$K[G]$-module? 
\item \textbf{Finding free bases and detecting primitivity:} Suppose we
are given a finite subset $B$ of $K[G]$ such that the right ideal
$I$ of $K\left[G\right]$ generated by $B$ is free as a right $K\left[G\right]$-module.
Are there algorithms for:
\begin{enumerate}
\item Finding a free basis for $I$ as a $K\left[G\right]$-module (as Rosenmann's
algorithm obtains for free group algebras)? 
\item Detecting whether a given element of $I$ is primitive (namely, part
of a basis) in $I$?
\end{enumerate}
\item \textbf{Computing intersections of one-sided ideals:} Given finite
subsets $B_{1}$ and $B_{2}$ in $K\left[G\right]$ which generate
the right ideals $I_{1}$ and $I_{2}$ respectively, is there an algorithm
to compute a generating set for $I_{1}\cap I_{2}$? \\
A solution of the analogous problem for free group algebras is given
in Algorithm \ref{alg:Intersection_of_modules}. For surface group
algebras, however, right ideals of $K\left[G\right]$ are not generally
free, so one may instead ask for such an algorithm under the additional
assumption that one or more of $I_{1},I_{2},I_{1}\cap I_{2}$ are
free $K\left[G\right]$-modules.
\end{enumerate}
\item \textbf{A canonical combinatorial object representing an ideal in
free group algebras:} Let $K$ be a field, $F$ a free group and $I\leq K\left[F\right]$
a right ideal. Can we associate to $I$ a combinatorial object which
plays a role analogous to the Stallings' core graph associated to
a subgroup of $F$? For a survey on algorithmic uses of Stalling core
graphs see \cite{Delgado2022}.
\end{enumerate}

\subsection*{Acknowledgements}

We would like to thank G.~Avramidi for referring us to the papers
\cite{Avramidi2022,Avramidi2023} and for suggesting the algorithmic
question above for surface groups. We are deeply grateful to the referees
for their careful reading of our manuscript, and in particular for
suggesting the addition of both the open question regarding the computation
of intersections in surface group algebras, as well as Remark \ref{rem: homological language}.\\
This work was supported by the European Research Council (ERC) under
the European Union\textquoteright s Horizon 2020 research and innovation
programme (grant agreement No 850956), by the Israel Science Foundation,
ISF grants 1140/23, as well as National Science Foundation under Grant
No. DMS-1926686.

\bibliographystyle{alpha}
\bibliography{duality_article_bib_library}

\begin{thebibliography}{EWPS24}

\bibitem[AD26]{Avramidi2023}
G.~Avramidi and T.~Delzant.
\newblock Group rings and hyperbolic geometry.
\newblock {\em Geom. Topol.}, 30(6):2337--2365, 2026.

\bibitem[Avr22]{Avramidi2022}
G.~Avramidi.
\newblock Division in group rings of surface groups.
\newblock {\em J. Lond. Math. Soc. (2)}, 106(2):982--1007, 2022.

\bibitem[AW92]{Adkins2012}
W.A. Adkins and S.H. Weintraub.
\newblock {\em Algebra: an approach via module theory}, volume 136 of {\em
  Princeton Landmarks in Mathematics and Physics}.
\newblock Springer Science \& Business Media, 1992.

\bibitem[BCS77]{Burns1977}
R.G. Burns, T.C. Chau, and D.~Solitar.
\newblock On the intersection of free factors of a free group.
\newblock {\em Proc. Amer. Math. Soc.}, 64(1):43--44, 1977.

\bibitem[CE56]{Cartan1999}
H.~Cartan and S.~Eilenberg.
\newblock {\em Homological algebra}, volume~28.
\newblock Princeton University Press, 1956.

\bibitem[Coh64]{Cohn1964}
P.M. Cohn.
\newblock Free ideal rings.
\newblock {\em J. Algebra}, 1(1):47--69, 1964.

\bibitem[Coh72]{Cohen2006}
D.E. Cohen.
\newblock {\em Groups of cohomological dimension one}, volume 245 of {\em
  Lecture Notes in Math.}
\newblock Springer, 1972.

\bibitem[Coh06]{Cohn2006}
P.M. Cohn.
\newblock {\em Free ideal rings and localization in general rings}, volume~3 of
  {\em New Mathematical Monographs}.
\newblock Cambridge University Press, 2006.

\bibitem[DV22]{Delgado2022}
J.~Delgado and E.~Ventura.
\newblock A list of applications of {S}tallings automata.
\newblock {\em Trans. Comb.}, 11(3):181--235, 2022.

\bibitem[EWPS24]{ErnstWest2024}
D.~Ernst-West, D.~Puder, and M.~Seidel.
\newblock Word measures on $\mathrm{GL}_{n}(q)$ and free group algebras.
\newblock {\em Algebra Number Theory}, 18(11):2047--2090, 2024.

\bibitem[Ger84]{Gersten1984}
S.M. Gersten.
\newblock On {W}hitehead's algorithm.
\newblock {\em Bull. Amer. Math. Soc. (N.S.)}, 10(2):281--284, 1984.

\bibitem[HA90]{HogAngeloni2006}
C.~Hog-Angeloni.
\newblock A short topological proof of {C}ohn's theorem.
\newblock In {\em Topology and Combinatorial Group Theory: Proceedings of the
  Fall Foliage Topology Seminars held in New Hampshire 1985-1988}, volume 1440
  of {\em Lecture Notes in Math.}, pages 90--95. Springer, 1990.

\bibitem[HP23]{Hanany2023}
L.~Hanany and D.~Puder.
\newblock Word measures on symmetric groups.
\newblock {\em Int. Math. Res. Not. IMRN}, 2023(11):9221--9297, 2023.

\bibitem[JZ24]{JaikinZapirain2024}
A.~Jaikin-Zapirain.
\newblock Free groups are ${L}^2$-subgroup rigid.
\newblock preprint arXiv:2403.09515, 2024.

\bibitem[KM02]{Kapovich2002}
I.~Kapovich and A.~Myasnikov.
\newblock Stallings foldings and subgroups of free groups.
\newblock {\em J. Algebra}, 248(2):608--668, 2002.

\bibitem[Lew69]{Lewin1969}
J.~Lewin.
\newblock Free modules over free algebras and free group algebras: the
  {S}chreier technique.
\newblock {\em Trans. Amer. Math. Soc.}, 145:455--465, 1969.

\bibitem[MP19]{Magee2019}
M.~Magee and D.~Puder.
\newblock Matrix group integrals, surfaces, and mapping class groups {I}: U(n).
\newblock {\em Invent. Math.}, 218:341--411, 2019.

\bibitem[MP22]{Magee2022}
M.~Magee and D.~Puder.
\newblock Core surfaces.
\newblock {\em Geometriae Dedicata}, 216(4), 2022.

\bibitem[MP24]{Magee2024}
M.~Magee and D.~Puder.
\newblock Matrix group integrals, surfaces, and mapping class groups {I}{I}:
  O(n) and {S}p(n).
\newblock {\em Math. Ann.}, 388(2):1437--1494, 2024.

\bibitem[MVW07]{Miasnikov2007}
A.~Miasnikov, E.~Ventura, and P.~Weil.
\newblock Algebraic extensions in free groups.
\newblock In {\em Geometric Group Theory: Geneva and Barcelona Conferences},
  Trends Math., pages 225--253. Springer, Birkh{\"a}user Basel, 2007.

\bibitem[PP15]{Puder2015}
D.~Puder and O.~Parzanchevski.
\newblock Measure preserving words are primitive.
\newblock {\em J. Amer. Math. Soc.}, 28(1):63--97, 2015.

\bibitem[PS26]{puder2023stable}
D.~Puder and Y.~Shomroni.
\newblock Stable invariants of words from random matrices.
\newblock {\em Forum Math. Sigma}, 2026.
\newblock With an appendix with D. Ernst-West and M. Seidel. To appear.
  Preprint available at arXiv:2311.17733.

\bibitem[Pud14]{Puder2014}
D.~Puder.
\newblock Primitive words, free factors and measure preservation.
\newblock {\em Israel J. Math.}, 201:25--73, 2014.

\bibitem[Ros93]{Rosenmann1993}
A.~Rosenmann.
\newblock An algorithm for constructing {G}r{\"o}bner and free {S}chreier bases
  in free group algebras.
\newblock {\em J. Symbolic Comput.}, 16(6):523--549, 1993.

\bibitem[Sei25]{Seidel2025}
M.~Seidel.
\newblock Exposure orders in free group algebras: Minimal {S}chreier
  transversals, free bases, and {G}r{\"o}bner bases.
\newblock {\em arXiv preprint arXiv:2510.05906}, 2025.

\bibitem[Sta68]{Stallings1968}
J.R. Stallings.
\newblock On torsion-free groups with infinitely many ends.
\newblock {\em Ann. Math.}, 88(2):312--334, 1968.

\bibitem[Sta83]{Stallings1983}
J.R. Stallings.
\newblock Topology of finite graphs.
\newblock {\em Invent. Math.}, 71(3):551--565, 1983.

\bibitem[Swa69]{Swan1969}
R.G. Swan.
\newblock Groups of cohomological dimension one.
\newblock {\em J. Algebra}, 12(4):585--610, 1969.

\bibitem[Umi94]{Umirbaev1994}
U.U. Umirbaev.
\newblock Primitive elements of free groups.
\newblock {\em Russian Math. Surveys}, 49(2):184, 1994.

\bibitem[Umi96]{Umirbaev1996}
U.U. Umirbaev.
\newblock On ranks of elements of free groups.
\newblock {\em Fundam. Prikl. Mat.}, 2(1):313--315, 1996.

\bibitem[Whi36a]{Whitehead1936}
J.H.C. Whitehead.
\newblock On certain sets of elements in a free group.
\newblock {\em Proc. Lond. Math. Soc.}, 2(1):48--56, 1936.

\bibitem[Whi36b]{Whitehead1936a}
J.H.C. Whitehead.
\newblock On equivalent sets of elements in a free group.
\newblock {\em Ann. Math.}, 37(4):782--800, 1936.

\end{thebibliography}

\end{document}